\documentclass{article}%
\usepackage{amssymb}
\usepackage{makeidx}
\usepackage[dvipsnames]{xcolor}
\usepackage[latin1]{inputenc}
\usepackage{tikz}
\usetikzlibrary{calc}
\usepackage{amsfonts,verbatim}
\usepackage{amsmath}
\usepackage{graphicx}
\usepackage{theorem}%
\providecommand{\U}[1]{\protect\rule{.1in}{.1in}}
\newtheorem{theorem}{Theorem}

\newtheorem{definition}[theorem]{Definition}
{\theorembodyfont{\rmfamily}

}

\newtheorem{lemma}[theorem]{Lemma}

{\theorembodyfont{\rmfamily}
\newtheorem{remark}[theorem]{Remark}
}

\newenvironment{proof}[1][Proof]{\noindent\textbf{#1} }{\ \rule{0.5em}{0.5em}}
\begin{document}

\author{O. Kneuss and W. Neves\\Instituto de Matem\'atica, Universidade Federal do Rio de Janeiro,\\ Cidade
Universit\'aria 21945-970,\\ Rio de Janeiro, Brasil\\olivier.kneuss@im.ufrj.br, wladimir@im.ufrj.br}
\title{Flows generated by divergence free 
vector fields
with compact support }
\maketitle

\begin{abstract}
We are concerned with the theory of existence and uniqueness of
flows generated by divergence free vector fields with compact support. 
Hence, assuming that the velocity 
vector fields are measurable, bounded, 
and the flows in 
the Euclidean space are measure preserving,
we show two counterexamples of uniqueness/existence  
for such flows. First we consider the autonomous case in dimension 3, and then, 
the non autonomous one in dimension 2.  
\end{abstract}

\section{Introduction}

We are concerned in this paper with the theory of existence and uniqueness of
flows generated by compactly supported,
divergence free vector fields. Moreover, we assume that the velocity 
vector fields are measurable, bounded, 
without differentiability regularity,
and the flows in 
the Euclidean space are measure preserving (with respect to Lebesgue measure).
Under these conditions we show two counterexamples of uniqueness/existence  
for such flows. First we consider the autonomous case in dimension 3, and then, 
the non autonomous one in dimension 2.  

\smallskip
The fundamental questions about the relation between 
velocity vector fields and flows 
come from long time ago with Lagrange, Euler, Bernoulli
among others important mathematicians.
In present-day it seems to be
reinitiated by Nelson \cite{Nelson} and 
put in more evidence by  Aizenman in his celebrated paper \cite{Aizenman}. 
This type of flows, as mentioned above, are encountered in 
many physical applications, for instance, related to fluid 
flow problems.

\smallskip
Although, one usually studies fluid dynamics using the Eulerian approach
instead of Lagrangean point of view given by the flow. This leaves 
to time evolution partial differential equations, in particular 
linear transport equations, which the uniqueness of weak solutions, 
for low regularity of the vector fields (called drift) has taken much attention. 
In this direction, we briefly recall the approach initiated in 1989 by DiPerna, Lions \cite{DiPerna-Lions},
where they proved uniqueness of weak solutions
for drift vector fields with Sobolev $W^{1,1}$ spatial regularity,
applying the nowadays well known commutators idea. Hence 
in 2004, Ambrosio \cite{Ambrosio} supported again on commutators, but
with a different measure-theoretic framework, 
extended the results of DiPerna, Lions 
for bounded variation drift vector fields. On 
those two papers, the uniqueness of the flow were obtained from the
uniqueness of the linear transport equation.  

\smallskip
Since Ambrosio's cited paper \cite{Ambrosio} there is a great effort to pass beyond 
BV vector fields. We remark that, 
the autonomous case in dimensions 2 is very particular 
(because the Hamiltonian structure),
and is completely understood. Indeed, it is proved in
\cite{A.B.C.Weak.Sard} a necessary and sufficient
condition for the uniqueness of bounded solutions of the linear transport equations, 
for bounded (divergence free) drifts
$a$; namely the Lipschitz potential $f$ of $a$ (i.e.
$a=(\partial_yf,-\partial_xf)$)
has to satisfy a ``weak'' Sard condition.
Moreover, it is
constructed in \cite{A.B.C.Structure.Lip} 
(see also Corollary 4.8 and its proof in \cite{A.B.C.Weak.Sard})  
a divergence free vector field $a$ with compact support
belonging to $C^{0,\alpha}(\mathbb{R}^2;\mathbb{R}^2)$ for every $\alpha<1$, 
for which the transport equation has
more than one solution.
Obviously this also provides a counterexample in dimensions three and higher (giving hence another proof of
Theorem \ref{theorem.supp.comp.R.3}). However it is not
known whether the vector field generates more that one regular flow (see definition below). 

\smallskip
The non uniqueness results established here are inspired by the strategies initiated by 
Aizenman \cite{Aizenman}, which is to say, to generate more than one flow 
from the same velocity vector field using low dimensional sets, see also Depauw \cite{Depauw}.
The precise description 
is made with details
in the following sections. 
Since the uniqueness of the linear transport equations implies 
uniqueness of the flow, as by product, our results implies non uniqueness of
the transport equations without the ``weak'' Sard property.  
This is very important for applications, let us mention two 
interesting open problems: The former one is the solvability of 
the Muskat problem, where the uniqueness (or renormalization) 
of the linear transport equations
with $L^2$ integrability of divergence free drift vector fields is an important
step towards the solution of this problem, see \cite{ChemetovNeves1, 
ChemetovNeves2}, and \cite{FNO} too. The second very interesting open problem is 
the wellposedness of the incompressible Euler's equations in dimension 3. Again, 
it is very important to know whenever
the renormalization property holds
for $L^2$ (divergence free) vector fields, see Lions' books \cite{lion1,lion2},
also De Lellis \cite{DeLellis}. 

\smallskip
We have sharpened the above two open problems,
with the counterexamples of uniqueness/renormalization established
in this paper. Albeit, it is not possible to close them yet, since the vorticity in
both problems should has some regularity, which is 
not the case in our examples. 

\subsection{Notation and Background}

At this point, we fix the notation used throughout the paper,
and recall some well known background. 

We denote by $\operatorname{div}$ the usual divergence operator.
Here $|\cdot|$ stands for the Lebesgue measure in $\mathbb{R}^n$, $(n=2, 3)$. 
Unless specified the contrary, 
any measure framework considered is respect
to Lebesgue measure.  

\begin{definition}
A family $\{\phi_t\}_{t \in \mathbb{R}}$, $\phi_t: \mathbb{R}^n \to \mathbb{R}^n$ of measurable maps is called
a measure preserving flow in $\mathbb{R}^n$, when it satisfies:

$(1)$ For each $t \in \mathbb{R}$, and every measurable set $A \subset \mathbb{R}^n$, 
$$
   |\phi_t^{-1}(A)|= |A|.
$$
The previous equation can be equivalently replaced by
$$
    \int_{\mathbb{R}^n}h(\phi_t(x)) \ dx
    = \int_{\mathbb{R}^n}h(y) \ dy, 
    \quad \text{for every $h\in L^1(\mathbb{R}^n)$}.
$$

$(2)$ For each $t_1, t_2 \in \mathbb{R}$, and a.e. $x \in \mathbb{R}^n$
$$
    \phi_{(t_1+t_2)}(x) = \phi_{t_1}(\phi_{t_2}(x)).
$$
\end{definition}

\begin{definition} \label{DEFLF}
Let $a(t,x)$ be a measurable vector field from $\mathbb{R} \times \mathbb{R}^n$ to $\mathbb{R}^n$, such that, 
$|a(t,x)| \leq \alpha(t)$ for some nonnegative function
$\alpha \in L^1_{\rm loc}(\mathbb{R})$. For each $T> 0$,
a mapping $\phi: [-T,T] \times \mathbb{R}^n \to \mathbb{R}^n$, $(\phi_t(\cdot) \equiv \phi(t,\cdot))$, 
is called a flow generated by the vector field $a(t,x)$, if for a.e. $x \in \mathbb{R}^n$,
the map $\phi(\cdot,x)$ is absolutely continuous in any compact subset of
$[-T,T]$, and satisfies
\begin{equation}
\label{ODEWS}
  \phi(t,x)= x + \int_0^t a\big(s,\phi(s,x)\big) \ ds.
\end{equation}

Moreover, we say that $\phi_t$ is regular if 
there exist positive constants $C, \tilde{C}$
(independent of $t$), such that, for each
Borel set $B \subset \mathbb{R}^n$
$$
  C \; |B| \leq \; \mu_{\phi_t}(B) \;
  \leq \tilde{C} \; |B|,
$$
where $\mu_{\phi_t}$ is the push-forward of the Lebesgue measure through the flow $\phi_t$. 
\end{definition}

\medskip
One remarks that, a necessary condition for a flow $\phi_t(\cdot)$ generated by $a(t,\cdot)$
be measure preserving is: $\operatorname{div} a= 0$ in a suitable sense. 

\section{The autonomous case}

\begin{theorem}
\label{theorem.supp.comp.R.3}
\textbf{Part 1: Non uniqueness.}
There exists a divergence free  vector field $a \in L^\infty(\mathbb{R}^3;\mathbb{R}^3)$
with compact support generating two distinct measure preserving flows satisfying the group property a.e..
More precisely, it will be shown the existence of two distinct measurable maps
$\phi,\psi:\mathbb{R}\times \mathbb{R}^3\rightarrow \mathbb{R}^3$
satisfying, for every $t \in \mathbb{R}$ and a.e $x \in \mathbb{R}^3$,
$$
    \phi(t,x)= x+\int_0^t a(\phi(s,x))ds,\quad \psi(t,x)= x + \int_0^ta(\psi(s,x))ds,
$$
such that $\phi(t,\cdot)$ and $\psi(t,\cdot)$ both preserve the Lebesgue
measure for every $t \in \mathbb{R}$ and such that, for a.e. $x\in \mathbb{R}^n$, for every 
$t_1\in \mathbb{R}$ and for every $t_2\in \mathbb{R}$ except a countable set (depending on $x$),
$$
    \phi(t_1+t_2,x)= \phi(t_1,\phi(t_2,x))
    \quad \text{and}\quad 
    \psi(t_1+t_2,x)= \psi(t_1,\psi(t_2,x)).
$$
Moreover, there exists a nontrivial
$L^{\infty}([0,\infty)\times \mathbb{R}^3)$ weak solution of
$$
    \partial_t u + \langle a;\nabla_x u\rangle= 0\quad \text{and}\quad u(0,\cdot)=0,
$$
which explicitly means that, for every $h\in C^{\infty}_c([0,\infty)\times \mathbb{R}^3)$
$$
    \int_{0}^{\infty}\int_{\mathbb{R}^3}  u(t,x) \, \left(\partial_t h(t,x)+\langle a(x);\nabla_x
    h(t,x)\rangle\right)dxdt=0.
$$
\textbf{Part 2: Non existence.}
There exists a compactly supported, divergence free  vector field $\tilde{a} \in L^\infty(\mathbb{R}^3;\mathbb{R}^3)$ 
generating no measure preserving flow satisfying the group property a.e..
\end{theorem}

The proof of the above result is inspired by \cite{Aizenman} and \cite{Depauw}.
The core idea is, as in \cite{Aizenman}, to construct a bounded divergence free vector field in $[0,1]^2\times (0,1]$ whose flow at some fixed time 
(here it will be $t=1$) collapses a large enough class of $1-$dimensional sets to points: That is, for $a.e.$ $x_2\in (0,1)$, the $x_1-$fiber $(0,1)\times \{x_2\}\times \{1\}$ is sent by the flow at time $1$ to a point in $(0,1)^2\times \{0\}.$ This will be done by following an argument in \cite{Depauw} using $2-$dimensional square and rectangle rotations: making use of such rotations we first exhibit a vector field whose flow at time $t=1/2$ sends, for $a.e.$ $x_2\in (0,1)$, the $x_1-$fiber $(0,1)\times \{x_2\}$ to a $x_1-$fiber of length $1/2$, then repeating the construction inductively (by scaling the geometry by a factor $1/2$) we finally obtain our desired vector field. Note also that a different  construction of a vector field with the same properties was done in \cite{C-L-R}.

Then, the vector field is extended to $\mathbb{R}^3$, so that, it remains bounded, divergence free, and has additionally compact support.

Using the above collapsing property we then construct, proceeding similarly as in \cite{Aizenman}, 
two distinct measure preserving flows $\phi$ and $\psi$ in $\mathbb{R}^3$ of our vector field which will be named $a$.
As a direct by-product we show that (as it would trivially be the case if $\phi$ and $\psi$ were smooth)
$u_0(\phi^{-1})$ and $u_0(\psi^{-1})$ both solve the linear transport equation with initial data $u_0$ and with drift term $a$. 
Choosing $u_0$ appropriately these two solutions are distinct which shows non-uniqueness for the transport equation.
Finally, by slightly modifying $a$,  we exhibit another vector field (with the same properties of $a$) for which there does not exist a measure preserving flow.

We stress on the fact that, all the bounded vector fields constructed in \cite{Aizenman} and  \cite{C-L-R}, resp. in
\cite{Depauw}, do not belong to $L^p(\mathbb{R}^3)$, resp. $L^p(\mathbb{R}^2)$, for any $p<\infty$ (and a fortiori are not bounded and with compact support).
Indeed, the vector fields \cite{Aizenman} and  \cite{C-L-R} are identically $(0,0,-1)$ in $(0,1)^2\times \big((-\infty,-1)\cup (1,\infty)\big)$ and the vector field constructed in \cite{Depauw} is periodic (with a square as period).

\begin{remark}
\label{remark.apres.theorme.supp.compact.R3} (i)
It is interesting to see that our vector fields $a$ and $\tilde{a}$ constructed below are moreover 
piecewise smooth in $\mathbb{R}^3\setminus([0,1]^2\times \{0\})$ (cf. Step 1.2 of the following proof).

(ii) Recall (proceeding for example by approximation) that, it always have existence of a (weak) bounded solution of the transport equation
$$\partial_t u+\langle a;\nabla_x u\rangle=0\quad \text{and} \quad u(0,\cdot)=u_0(\cdot)$$ when $a$ and $u_0$ are bounded.

(iii) As a direct consequence of the a.e. group property (cf. Step 6.3 in the proof below) we will also show that, for 
every $t\in \mathbb{R}$, $\phi(t,\cdot)$ and $\psi(t,\cdot)$ both are bijection from an open set of full 
measure in $\mathbb{R}^3$ onto an open set of full measure in $\mathbb{R}^3$ (depending on $t$) and that,
$$\phi(t,\cdot)^{-1}=\phi(-t,\cdot)\quad \text{and}\quad \phi(t,\cdot)^{-1}=\phi(-t,\cdot).$$
\end{remark}

\begin{proof} The proof is organized as follows. In the first 6 steps we establish the non uniqueness for the flow. 
In Step 7 we prove the non uniqueness for the transport equation. Finally in Step 8 we show the non existence part.

\textit{\underline{Step 1:} Definition of the vector field $a$ and its properties.}

\textit{\underline{Step 1.1.}}
The measurable and bounded vector field $a(x)= a(x_1,x_2,x_3)$,
with compact support and divergence free, will be first defined in the upper 
half space and then in the lower half space. For its definition we will use two vector fields exhibited in the appendix.

Define $a$ in $\{x_3\geq 0\}$ by
$$
a(x_1,x_2,x_3):=\left\{
  \begin{aligned}
  (b(1-x_3,x_1,x_2),-1)& \quad \text{in $A_1$,}
   \\[5pt]
   \frac{(0,x_3-1,-x_2-1)}
   {\sqrt{(x_3-1)^2+(x_2+1)^2}}& \quad \text{in $A_2$,}
   \\[5pt]
   \left(\frac{c(x_1-1/2,x_2+5/2)}{2},1\right) & \quad \text{in $A_3$,}
    \\[5pt]
   0 \quad \quad & \quad \text{in $\{x_3\geq 0\}\setminus (A_1\cup A_2\cup A_3)$,}
  \end{aligned}
 \right.
$$
where (cf. Figure 1)
\begin{align*}A_1&:=[0,1]^2\times (0,1]\\
A_2&:=\{(x_1,x_2,x_3)\in \mathbb{R}^3|\text{ }0\leq x_1\leq 1,1\leq \sqrt{(x_2+1)^2+(x_3-1)^2}\leq 2,x_3\geq 1\}\\
A_3&:=[0,1]\times [-3,-2]\times [0,1],
\end{align*}
$b:(-\infty,1)\times \mathbb{R}^2\rightarrow \mathbb{R}^2$ is the 2-dimensional vector field  defined in Lemma \ref{lemma.vector.field.b} and 
$c:\mathbb{R}^2\rightarrow \mathbb{R}^2$ is the 2-dimensional autonomous  vector field  defined in Lemma \ref{lemma.rotation.carre} (i).
We next define $a$ in $\{x_3<0\}$ as follows:
$$
a(x_1,x_2,x_3):=\left\{
  \begin{aligned}
  (0,0,-1)& \quad \text{in $A_4$,}
   \\[5pt]
   \frac{(0,x_3+2,-x_2-1)}
   {\sqrt{(x_3+2)^2+(x_2+1)^2}}& \quad \text{in $A_5$,}
   \\[5pt]
   \left(\frac{-c(x_1-1/2,x_2-1/2)}{2},-1\right) & \quad \text{in $A_6$,}
    \\[5pt]
    -R_3(a(R_3(x))) & \quad \text{in $A_7$,}
    \\[5pt]
   0 \quad \quad & \quad \text{in $\{x_3< 0\}\setminus (A_4\cup A_5\cup A_6\cup A_7)$,}
  \end{aligned}
 \right.
$$
where (cf. Figure 1) \begin{align*}
A_4&:=[0,1]\times [-3,-2]\times [-2,0]\\
A_5&:=\{(x_1,x_2,x_3)\in \mathbb{R}^3|\text{ }0\leq x_1\leq 1,1\leq \sqrt{(x_2+1)^2+(x_3+2)^2}\leq 2,x_3\leq -2\}\\
A_6&:=[0,1]^2\times [-2,-1]\\
A_7&:=[0,1]^2\times [-1,0)
\end{align*}
where $c$ is as before the vector field defined in Lemma \ref{lemma.rotation.carre} (i)
 and
where $$R_3(x_1,x_2,x_3):= (x_1,x_2,-x_3).$$
The definition of $a$ in $\{x_3<0\}$ might not appear to be the most natural one (one could have defined it by reflection everywhere in the lower half space e.g.); 
however with the definition the "period" of the flow of $a$ will be (contrary to the definition by reflection) independent of the position (cf. \eqref{period.indep.of.x}) which will significantly simplify some technical parts of the present proof.
\medskip

\textit{\underline{Step 1.2:} Properties of $a$.}
Let
$$S:=\cup_{i=1}^7A_i.$$
First, since $a\equiv 0$ outside $\overline{S}$ and $S$ is a bounded set,
the vector field $a$ has compact support (cf. Figure 1 for a representation of $a$).
Next, since from Lemma \ref{lemma.vector.field.b},
$b\in L^{\infty}([0,1)\times \mathbb{R}^2)$ and since (cf. Lemma \ref{lemma.rotation.carre}) $c$ is bounded, we directly get that $a$ in bounded in $\mathbb{R}^3$.
Using in particular the definition of the vector fields $b$ and $c$ we directly get that $a$ is piecewise smooth in $\mathbb{R}^3\setminus([0,1]^2\times \{0\})$: there exist countably pairwise disjoint open sets
$U_i$ with the following properties:
\begin{itemize}
\item $a$ is smooth in every $U_i$ and can be extended in a smooth way to $\overline{U}_i$
\item $\bigcup_i \overline{U}_i=\mathbb{R}^3\setminus([0,1]^2\times \{0\})$
\item for every $x\in \mathbb{R}^3\setminus([0,1]^2\times \{0\})$ we can find a neighbourhood of $x$
    intersecting only finitely many $\overline{U}_i's$.
\end{itemize}
In fact, except for finitely many i's, the $U_i's$ will be of the form $T_i\times I_i$ where $T_i$ is an 
open isosceles triangle in $\mathbb{R}^2$ and $T_i$ is an open interval in $\mathbb{R}.$

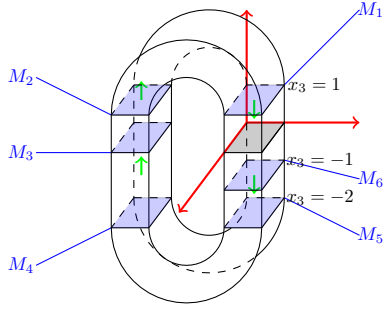
\begin{figure}
\begin{center}
\begin{tikzpicture}[scale=0.5]
\draw[blue](1,1)--(3,3);\draw[blue](1,-1)--(3,-1.5);\draw[blue](1,-2)--(3,-3);
\draw (3.3,3) node[scale=0.7,blue]{$M_1$};\draw (3.3,-1.5) node[scale=0.7,blue]{$M_6$};\draw (3.3,-3) node[scale=0.7,blue]{$M_5$};
\draw[blue](-3-3/5,1-4/5)--(-5-3/5,2-4/5);\draw[blue](-3-3/5,-4/5)--(-5-3/5,-4/5);\draw[blue](-3-3/5,-2-4/5)--(-5-3/5,-3-4/5);
\draw (-5-3/5-0.4,2-4/5) node[scale=0.7,blue]{$M_2$};\draw (-5-3/5-0.4,-4/5) node[scale=0.7,blue]{$M_3$};\draw (-5-3/5-0.4,-3-4/5) node[scale=0.7,blue]{$M_4$};
\draw (1.8,1) node[scale=0.7]{$x_3=1$};
\draw (1.95,-1) node[scale=0.7]{$x_3=-1$};
\draw (1.95,-2) node[scale=0.7]{$x_3=-2$};
\draw[arrows=->,thick,red](0,0)--(0,3);\draw[arrows=->,thick,red](0,0)--(3,0);
\draw[arrows=->,thick,red](0,0)--(-3*3/5,-3*4/5);
\draw[arrows=->,green,thick] (-2.5-3/10,1-4/10)--(-2.5-3/10,-4/10+1.5);
\draw[arrows=->,green,thick] (-2.5-3/10,-1-4/10)--(-2.5-3/10,-4/10+-0.5);
\draw[arrows=->,green,thick] (0.5-3/10,-4/10+1)--(0.5-3/10,0.5-4/10);
\draw[arrows=->,green,thick] (0.5-3/10,-4/10-1)--(0.5-3/10,-1.5-4/10);
\draw[dashed](0,0)--(0,-1);
\draw (-3/5,-4/5-1)--(-3/5,-4/5+1);\draw (1-3/5,-4/5-1)--(1-3/5,-4/5+1);
\draw (1,-1)--(1,1); \draw[dashed] (1,1)--(0,1);
\draw (1,1)--(1-3/5,-4/5+1);\draw (1,0)--(1-3/5,-4/5);
\draw (1,0)--(1-3/5,-4/5);\draw (1,-1)--(1-3/5,-4/5-1);
\draw (-3/5,-4/5)--(-3/5+1,-4/5);\draw (-3/5,-4/5+1)--(-3/5+1,-4/5+1);
\draw [dashed](-3/5,-4/5+1)--(0,1); \draw [dashed](0,-1)--(1,-1);
\draw[dashed] (0,-1)--(-3/5,-4/5-1);
\draw(-3/5,-4/5-1)--(-3/5+1,-4/5-1);
\draw [domain=0:180,dashed] plot ({-1+cos(\x)}, {1+sin(\x)});
\draw [domain=0:158] plot ({-1+2*cos(\x)}, {1+2*sin(\x)});
\draw [domain=158:180,dashed] plot ({-1+2*cos(\x)}, {1+2*sin(\x)});
\draw [domain=0:180] plot ({-1-3/5+cos(\x)}, {1-4/5+sin(\x)});
\draw [domain=0:180] plot ({-1-3/5+2*cos(\x)}, {1-4/5+2*sin(\x)});
\draw [domain=180:294] plot ({-1+cos(\x)}, {-2+sin(\x)});
\draw [domain=294:360,dashed] plot ({-1+cos(\x)}, {-2+sin(\x)});
\draw [domain=180:308,dashed] plot ({-1+2*cos(\x)}, {-2+2*sin(\x)});
\draw [domain=308:360] plot ({-1+2*cos(\x)}, {-2+2*sin(\x)});
\draw [domain=180:360] plot ({-1-3/5+cos(\x)}, {-2-4/5+sin(\x)});
\draw [domain=180:360] plot ({-1-3/5+2*cos(\x)}, {-2-4/5+2*sin(\x)});
\draw[dashed] (-3,1)--(-2,1);
\draw [dashed](-3,-2)--(-2,-2);
\draw[dashed] (-3,1)--(-3,-2);\draw (-2,1)--(-2,-2);
\draw [dashed](-3,1)--(-3-3/5,1-4/5);\draw (-2,1)--(-2-3/5,1-4/5);
\draw [dashed](-3,0)--(-3-3/5,0-4/5);\draw (-2,0)--(-2-3/5,-4/5);
\draw [dashed](-3,-2)--(-3-3/5,-2-4/5);\draw (-2,-2)--(-2-3/5,-2-4/5);
\draw [dashed](-3,0)--(-2,0);
\draw (-3-3/5,-4/5)--(-2-3/5,-4/5);
\draw (-3-3/5,1-4/5)--(-2-3/5,1-4/5);
\draw (-3-3/5,-2-4/5)--(-2-3/5,-2-4/5);
\draw (-3-3/5,1-4/5)--(-3-3/5,-2-4/5);\draw (-2-3/5,1-4/5)--(-2-3/5,-2-4/5);
\draw [dashed] (-3/5,-2-4/5)--(0,-2)--(1,-2);
\draw (-3/5,-4/5-2)--(-3/5+1,-4/5-2)--(1,-2);
\draw [dashed] (0,-1)--(0,-2);
\draw (1,-1)--(1,-2);\draw (-3/5+1,-4/5-1)--(-3/5+1,-4/5-2);\draw (-3/5,-4/5-1)--(-3/5,-4/5-2);
\draw[fill=black, opacity=0.2] (0,0)--(1,0)--(1-3/5,-4/5)--(-3/5,-4/5);
\draw[fill=blue, opacity=0.2] (0,1)--(1,1)--(1-3/5,1-4/5)--(-3/5,1-4/5);
\draw[fill=blue, opacity=0.2] (0,-1)--(1,-1)--(1-3/5,-1-4/5)--(-3/5,-1-4/5);
\draw[fill=blue, opacity=0.2] (0,-2)--(1,-2)--(1-3/5,-2-4/5)--(-3/5,-2-4/5);
\draw[fill=blue, opacity=0.2] (-3,0)--(-2,0)--(-2-3/5,-4/5)--(-3-3/5,-4/5);
\draw[fill=blue, opacity=0.2] (-3,1)--(-2,1)--(-2-3/5,1-4/5)--(-3-3/5,1-4/5);
\draw[fill=blue, opacity=0.2] (-3,-2)--(-2,-2)--(-2-3/5,-2-4/5)--(-3-3/5,-2-4/5);
\end{tikzpicture}
\caption{A representation of $a$:
The black subset at $\{x_3=0\}$ is $M_0=M_7$ while the blue subsets represent $M_1,\cdots M_6$ and are enumerated counter clockwise starting at $M_0$. For $i=1,\cdots 7,$ $A_i$ is the region delimited by $M_{i-1}$ and $M_i$:
$A_1=[0,1]^2\times (0,1]$ and so on until $A_7=[0,1]^2\times [-1,0).$
The whole donut (without the black subset) is the union of the $A_i$'s and is referred to as $S.$ The four green arrows represent roughly the direction $a$.}\end{center}
\end{figure}

\medskip
We now show that $\operatorname{div}(a)=0$ in $\mathbb{R}^3$ in the sense of distributions.
First since $b(t,\cdot)$ is divergence free in $(-1/2,1/2)^2$ for every $t\in [0,1)$  we directly get that
$\operatorname{div}a= 0$ in $A_1$ and in $A_7$.
Similarly, since (cf. Lemma \ref{lemma.rotation.carre}) $c$ is divergence free in $(-1/2,1/2)^2$ we get that
$\operatorname{div}a= 0$ in $A_3$ and in $A_6$.
Moreover, we trivially have that $\operatorname{div}a= 0$ in $A_2$, $A_3$ and $A_5.$
Next, noting the normal component of $a$ is continuous across every horizontal 
component of $\cup_{i=1}^7\partial A_i$ (of course the normal component of $a$ is the third component $a$ on such components)  we directly get that $\operatorname{div}a= 0$ in $S\cup \big((0,1)^2\times \{0\}\big).$
Finally since obviously $\operatorname{div}a= 0$ in $\mathbb{R}^3\setminus \{S\cup (0,1)^2\times \{0\}\}$, and since, using in particular Lemmas \ref{lemma.vector.field.b} and \ref{lemma.rotation.carre} (i),
the normal component of $a$ is zero (and hence continuous) across every not horizontal part of $\cup_{i=1}^7\partial A_i$ we get that $\operatorname{div}a=0$ in $\mathbb{R}^3$ as wished.

\medskip
\textit{\underline{Step 2:} Definition of a measure preserving flow of $a$ up to some positive and negative stopping times.}
In this step we prove that, for every $x\in S$, there exist some finite positive time $t^{+}(x)$ and some finite
negative time $t^{-}(x)$ and a measurable map $\varphi(t,x)$ defined for $t\in [t^{-}(x),t^{+}(x)]$  
with the following properties:\begin{itemize}\item Flow of $a$ in $S$:
for every $x\in S$
\begin{equation}
\label{varphi.flot.on.S}
   \varphi(t,x)= x + \int_0^ta(\varphi(s,x))ds
   \quad \text{for $t\in [t^{-}(x),t^{+}(x)],$}
\end{equation}
\begin{equation}\label{varphi.dans.S}
   \varphi(t,x) \in S \quad \text{for $t\in (t^-(x),t^+(x))$}
\end{equation} 
and
\begin{equation}\label{varphi.dans.plan.critique}
  \varphi(t^{\pm}(x),x) \in [0,1]^2\times \{0\}.
\end{equation}
\item Group property: for every $x \in S$ and $t_1,t_2\in \mathbb{R}$, such that
 $t_2\in (t^{-}(x),t^{+}(x))$ and $t_1+t_2\in [t^{-}(x),t^{+}(x)]$ we have
\begin{equation}
\varphi(t_1+t_2,x)= \varphi(t_1,\varphi(t_2,x))
\label{semi.group.varphi.x}
\end{equation} 
and
\begin{equation}
 \label{semi.group.varphi.time.t.pm}
   t^{\pm}(\varphi(t_2,x))= t^{\pm}(x)-t_2.
\end{equation}
\item Measure preservation:
 for every $t\in \mathbb{R}$  and every measurable set $U\subset S$, such that, $t\in (t^-(x),t^{+}(x))$ for every $x\in U$ then
\begin{equation}\label{varphi.preserve.measure}\varphi(t,\cdot)|_{U}:U\rightarrow \varphi(t,U)\text{ preserves the measure.}
\end{equation}
\item Local bijectivity:  for every $t\in \mathbb{R}$, and every set $U\subset S$, such that, $t\in (t^-(x),t^{+}(x))$ for every $x\in U$ then
\begin{equation}\label{varphi.injectivity}\varphi(t,\cdot)|_{U}:U\rightarrow \varphi(t,U)\quad\text{is bijective.}
\end{equation}
\end{itemize}
In words (cf. \eqref{varphi.dans.S} and \eqref{varphi.dans.plan.critique}) $t^{+}(x)$, resp. $t^{-}(x)$, is the smallest positive time, resp. the biggest negative time, after which the flow $\varphi(\cdot,x)$ reaches the plane $[0,1]^2\times\{0\}$
from above, resp. from below. Recall that, if $t \in (t^{-}(x),t^{+}(x))$ 
then $\varphi(t,x) \in S$ (and hence does not belong to $[0,1]^2 \times \{0\}$). 

\smallskip
The idea for the construction of $\varphi$ and $t^{\pm}$ is elementary: recalling that $S=\cup_{i=1}^7A_i$ we first exhibit, for $i=1,\cdots,7,$
times $t^{\pm}_{i}:A_i\rightarrow \mathbb{R}$ and a flow $\varphi$ in $A_i$ satisfying \eqref{varphi.flot.on.S},\eqref{semi.group.varphi.x}-\eqref{varphi.injectivity} (with $S$ replaced by $A_i$ and with $t^{\pm}$ replaced by $t^{\pm}_{i}$). See Figure 2 for an illustration of $t^{\pm}.$
Denoting (cf. Figure 1)
$$M_0= M_7:=[0,1]^2 \times \{0\}, \quad M_1:= [0,1]^2\times \{1\},\quad M_2:=[0,1]\times [-3,-2]\times \{1\}$$
$$M_3:=[0,1]\times [-3,-2]\times \{0\},\quad M_4:=[0,1]\times [-3,-2]\times \{-2\},\quad M_5:=[0,1]^2\times \{-2\}$$
$$M_6:=[0,1]^2\times \{-1\}.$$
we will also have that, for every $1\leq i\leq 7,$
\begin{equation}\label{varphi.i.dans.A.i}
   \varphi(t,x)\in A_i\setminus (M_{i-1}\cup M_i)\quad \text{for $t\in (t^{-}_{i}(x),t^{+}_{i}(x))$ and $x\in A_i$},
\end{equation}
\begin{equation}\label{varphi.i.t.pm.i}
   \varphi(t^{-}_{i}(x),x) \in M_i \quad \text{and} \quad \varphi(t^{+}_{i}(x),x) \in M_{i-1}\quad \text{for $x\in A_i$}.
\end{equation}
It will hence be possible to glue the orbits on $A_i$ and obtain our desired flow $\varphi$ as well as $t^{\pm}$.\smallskip

\begin{itemize}
\item Flow of $a$ in $A_1$: Define for every $x\in A_1$ and every $$t\in [x_3-1,x_3]=:[t^{-}_{1}(x),t^{+}_{1}(x)]$$
$$
    \varphi(t,x):=(\chi^{(1-x_3)}(t,x_1,x_2),x_3-t),
$$
where $\chi^{(1-x_3)}$ is the flow of $(t,x_1,x_2)\rightarrow b(t+1-x_3,x_1,x_2)$ exhibited in Lemma \ref{lemma.flow.of.b}. By the properties of $\chi^{(\cdot)}$ 
listed in Lemma \ref{lemma.flow.of.b}, it is a simple exercise to check that $\varphi$ satisfies \eqref{varphi.i.dans.A.i}, \eqref{varphi.i.t.pm.i} and
\eqref{varphi.flot.on.S}, \eqref{semi.group.varphi.x}-\eqref{varphi.injectivity} with $S$ replaced by $A_1$ and
$t^{\pm}$ replaced by $t^{\pm}_{1}$.\smallskip

\item Flow of $a$ in $A_2$. Define for every $x\in A_2$, writing $x=(x_1,r\cos(\theta)-1,r\sin(\theta)+1)$ with $r\in [1,2]$ and $\theta\in [0,\pi]$, and every $$t\in [r(\theta-\pi),r\theta]=:[t^{-}_{2}(x),t^{+}_{2}(x)],$$
$$\varphi(t,x):=(x_1,r\cos(\theta-t/r)-1,r\sin(\theta-t/r)+1).$$ 
It is elementary to check that $\varphi$ satisfies \eqref{varphi.i.dans.A.i}, \eqref{varphi.i.t.pm.i} and \eqref{varphi.flot.on.S}, \eqref{semi.group.varphi.x}-\eqref{varphi.injectivity} with $S$ replaced by $A_2$ and
$t^{\pm}$ replaced by $t^{\pm}_{2}.$
In particular, note that for every $x \in [0,1]^2\times \{1\}$, then $t^-_2(x)=-x_2-1$ and
\begin{equation}\label{formule.flow.A.2}
\varphi(-x_2-1;x)=(x_1,-x_2-2,1).
\end{equation}

\item Flow of $a$ in $A_3$: Define for every $x\in A_3$ and $$t\in [-x_3,1-x_3]=:[t^{-}_{3}(x),t^{+}_{3}(x)]$$
$$\varphi(t,x):=\big(\xi^c(t/2,x_1-1/2,x_2+5/2)+(1/2,-5/2),x_3+t\big)$$ where $\xi^c$ is the flow exhibited is Lemma \ref{lemma.rotation.carre}. It is easy to check that  $\varphi$ satisfies \eqref{varphi.i.dans.A.i}, \eqref{varphi.i.t.pm.i} and \eqref{varphi.flot.on.S}, \eqref{semi.group.varphi.x}-\eqref{varphi.injectivity} with $S$ replaced by $A_3$ and
$t^{\pm}$ replaced by $t^{\pm}_{3}$. In particular note that for every $x\in [0,1]\times [-3,-2]\times \{1\}$ then $t_3^-(x)=-1$ and
\begin{equation}\label{formule.flow.A.3}
\varphi(-1,x)=(-x_1+1,-x_2-5,0).
\end{equation}

\item Flow of $a$ in $A_4$: Define for every $x\in A_4$ and $$t \in [-x_3-2,-x_3] =:[t^{-}_{4}(x),t^{+}_{4}(x)]$$
$$\varphi(t,x):=(x_1,x_2,x_3+t).$$ Trivially, since $a=(0,0,1)$ in $A_4,$ $\varphi$ satisfies \eqref{varphi.i.dans.A.i}, \eqref{varphi.i.t.pm.i} and \eqref{varphi.flot.on.S}, \eqref{semi.group.varphi.x}-\eqref{varphi.injectivity} with $S$ replaced by $A_4$ and
$t^{\pm}$ replaced by $t^{\pm}_{4}$.\smallskip

\item Flow of $a$ in $A_5$. Define for every $x\in A_5$ writing $x= (x_1,r\cos(\theta)-1,r\sin(\theta)-2)$ with $r \in [1,2]$ and $\theta\in [\pi,2\pi]$, 
and every $$t\in [r(\theta-2\pi),r(\theta-\pi)]=:[t^{-}_{5}(x),t^{+}_{5}(x)],$$
$$\varphi(t,x):= (x_1,r\cos(\theta-t/r)-1,r\sin(\theta-t/r)-2).$$ 
As before it is elementary to check that $\varphi$ satisfies \eqref{varphi.i.dans.A.i}, \eqref{varphi.i.t.pm.i} and \eqref{varphi.flot.on.S}, \eqref{semi.group.varphi.x}-\eqref{varphi.injectivity} with $S$ replaced by $A_5$ and
$t^{\pm}$ replaced by $t^{\pm}_{5}.$
In particular note that for every $x\in [0,1]\times [-3,-2]\times \{-2\}$ then $t_5^{-}(x)=x_2+1$ and
\begin{equation}\label{formule.flow.A.5}
\varphi(x_2+1,x)= (x_1,-x_2-2,-2).
\end{equation}

\item Flow of $a$ in $A_6$: Define for every $x\in A_6$ and $$t\in [1+x_3,2+x_3]=:[t^{-}_{6}(x),t^{+}_{6}(x)]$$
$$\varphi(t,x):= \big(\xi^c(-t/2,x_1-1/2,x_2-1/2)+(1/2,1/2),x_3-t\big)$$ where $\xi^c$ is the flow exhibited is Lemma \ref{lemma.rotation.carre}. It is easy to check that  $\varphi$ satisfies \eqref{varphi.i.dans.A.i}, \eqref{varphi.i.t.pm.i} and \eqref{varphi.flot.on.S}, \eqref{semi.group.varphi.x}-\eqref{varphi.injectivity} with $S$ replaced by $A_6$ and
$t^{\pm}$ replaced by $t^{\pm}_{6}$. In particular note that for every $x\in [0,1]^2\times \{-2\}$ then $t_6^{-}(x)=-1$ and
\begin{equation}\label{formule.flow.A.6}
\varphi(-1,x)= (1-x_1,1-x_2,-1).
\end{equation}

\item Flow of $a$ in $A_7.$ Define for every $x\in A_7$ and every $$t\in[x_3,1+x_3]=:[t^{-}_{7}(x),t^{+}_{7}(x)],$$
\begin{equation}\label{formule.flow.A.7}\varphi(t,x):= R_3(\varphi(-t,R_3(x))).\end{equation}
Since $a$ has been defined by reflection on $A_7= R_3(A_1)$, i.e. $$a(x)=-R_3(a(R_3(x))),$$ 
combining Lemma \ref{lemma.reflection} and the flow constructed in $A_1$ we immediately get that
$\varphi$ satisfies \eqref{varphi.i.dans.A.i}, \eqref{varphi.i.t.pm.i} and \eqref{varphi.flot.on.S}, \eqref{semi.group.varphi.x}-\eqref{varphi.injectivity} with $S$ replaced by $A_7$ and
$t^{\pm}$ replaced by $t^{\pm}_{7}$.
\end{itemize}

Then, we naturally define $t^{\pm}$ as follows:
For $x\in A_i$, we set (cf. \eqref{varphi.i.t.pm.i}) $y_{i}^+(x)=\varphi(t^{+}_{i}(x),x) \in M_{i-1}$ and
$y_{i}^-(x)=\varphi(t^{-}_{i}(x),x) \in M_{i}$.
For every $1\leq j<i$ define by induction
$$y_{j}^{+}(x):=\varphi(t^{+}_{j}(y_{{j+1}}^{+}(x)),y_{{j+1}}^{+}(x))\in M_{j-1}$$ and similarly for every $i<l\leq 7$,
$$y_{l}^{-}(x):=\varphi(t^{-}_{l}(y_{{l-1}}^{-}(x)),y_{{l-1}}^{-}(x))\in M_{l}.$$
Then define
$$t^+(x):=t^{+}_{i}(x)+\sum_{1\leq j<i}t^{+}_{j}(y_{{j+1}}^{+}(x))$$ and
$$t^-(x):=t^{-}_{i}(x)+\sum_{i<l\leq 7}t^{-}_{l}(y_{{l-1}}^{-}(x)).$$
Finally, we obtain our desired $\varphi(t,x)$ for $x\in S$ and $t\in [t^{-}(x),t^{+}(x)]$  by gluing the orbits of the previously obtained flows on $A_i$.
Note in particular that \eqref{varphi.preserve.measure} is satisfied since $a$ is divergence free.
Note also that, since the third component of $\varphi(t,x)$ is $x_3-t$ for $x\in [0,1]^2\times (0,1]$ and $t\in [x_3-1,x_3]$, we directly get from \eqref{varphi.dans.plan.critique} that
\begin{equation}\label{t.plus.x.moins.1}
\varphi(t^{+}(x)-1,x)\in [0,1]^2\times \{1\}\quad \text{for every $x\in S.$}
\end{equation}

\textit{\underline{Step 3:} Additional properties of $\varphi$ and $t^{\pm}$}

\begin{itemize}
\item Recalling that $a$ is piecewise smooth in $\mathbb{R}^3\setminus ([0,1]^2\times \{0\})$ we get in particular $a\in BV(S)$. 
Hence (cf. \cite{Ambrosio}), $\varphi$ is the unique measure preserving flow (up to a null set) of $a$ in $S$.

\item Noting that $t^{\pm}_{A_i}$ is continuous in $A_i$ and does not depend of $x_1$ we deduce that the same holds for $t^{\pm}$ namely:
\begin{equation}\label{times.ne.dependent.pas.de.x1}
t^{\pm}\quad\text{does not depend on $x_1$ an is continuous on $S$.}
\end{equation}
Moreover it is easily checked that
\begin{equation}\label{t.plus.moins.1.measure}
|\{x\in S|\text{ }t^{+}(x)=t\}|=0\quad \text{for every $t\in \mathbb{R}$}.
\end{equation}

\item For every $x\in S$ we claim that
\begin{equation}\label{period.indep.of.x}
t^{+}(x)-t^-(x)=6+3\pi\end{equation}
and is hence independent of $x.$
Indeed using first \eqref{semi.group.varphi.time.t.pm} we get that for every $x\in S$
$$t^{+}(\varphi(t^{+}(x)-1,x))-t^-(\varphi(t^{+}(x)-1,x))= t^{+}(x)-t^{-}(x);$$
hence, using \eqref{t.plus.x.moins.1}, it is sufficient to prove to claim for $x\in [0,1]^2\times \{1\}=M_1.$
Then note that $x\in M_1$ is sent by $\varphi$ to $[0,1]^2\times \{0\}$ after a time $t=1$, hence
$t^{+}(x)=1.$ Next, using \eqref{formule.flow.A.2}, $x$ is sent by $\varphi$ to $(x_1,-x_2-2,1)\in M_2$ after a time $t=-\pi(x_2+1)$.
By \eqref{formule.flow.A.3} $\varphi$ sends then $(x_1,-x_2-2,1)$ to $(1-x_1,x_2-3,0)\in M_3$ after a time $-1$.
Trivially $(1-x_1,x_2-3,0)$ is sent by $\varphi$ to $(1-x_1,x_2-3,-2)\in M_4$ after a time $t=-2.$
Using \eqref{formule.flow.A.5} $\varphi$ sends $(1-x_1,x_2-3,-2)$ to $(1-x_1,1-x_2,-2)\in M_5$ after a $t=-\pi(2-x_2).$
From \eqref{formule.flow.A.6} $\varphi$ sends $(1-x_1,1-x_2,-2)$ to $(x_1,x_2,-1)\in M_6$ after a time $t=-1$
and finally $(x_1,x_2,-1)$ is sent by $\varphi$ in $[0,1]^2\times \{0\}=M_7$ after a time $t=-1.$
So at the end
$$t^-(x)=-\pi(x_2+1)-1-3-\pi(2-x_2)-1-1=-5-3\pi$$ and therefore
$$t^{+}(x)-t^{-}(x)=6+3\pi$$
as claimed.
Note that in particular it has been shown that for every $x\in M_1$
\begin{equation}\label{t.moins.plus.1}\varphi(t^{-}(x)+1;x)=(x_1,x_2,-1).\end{equation}

\item Periodicity of $\varphi:$ We claim that, for every $x\in S,$
\begin{equation}\label{varphi.extreme.equale}\varphi(t^{+}(x),x)= \varphi(t^{-}(x),x)\in [0,1]^2\times
\{0\}.\end{equation}
As before, using \eqref{semi.group.varphi.x} and \eqref{semi.group.varphi.time.t.pm}
we get that
$$\varphi(t^{+}(x),x)= \varphi(1,\varphi(t^{+}(x)-1,x)),$$ 
$$\varphi(t^{-}(x),x)= \varphi(t^{-}(x)-t^{+}(x)+1,\varphi(t^{+}(x)-1,x))$$
and
$$1=t^{+}(\varphi(t^{+}(x)-1,x))\quad \text{and}\quad t^{-}(x)-t^{+}(x)+1= t^{-}(\varphi(t^{+}(x)-1,x)).$$
Hence from \eqref{t.plus.x.moins.1}, it is enough to prove \eqref{varphi.extreme.equale} when $x\in [0,1]^2\times \{1\}.$ For such  $x$ we have 
$\varphi(t^{+}(x);x)=\varphi(1,x)$. Using \eqref{semi.group.varphi.x}, \eqref{t.moins.plus.1} and \eqref{formule.flow.A.7} we hence get, by definition of $\varphi$ in $A_7$ (cf. Step 2)
\begin{align*}\varphi(t^{-}(x),x)&=\varphi(-1,\varphi(t^{-}(x)+1,x))=\varphi(-1,x_1,x_2,-1)\\
&=\varphi(1,x_1,x_2,1)= \varphi(t^+(x),x)\end{align*}
as claimed.

\item Collapsing of $x_1-$fibers:
We claim that \begin{equation}\label{coll.2d}
\varphi(1,(0,1)\times \{x_2\} \times \{1\}) \quad \text{is a singleton in $(0,1)^2\times \{0\}$}
\end{equation} for every $x_2\in (0,1)\setminus Z$
where
$$Z:=\left\{\frac{j}{2^{i}}|\text{ }0\leq j\leq 2^i,i\geq 1\right\}.$$
It means that, except for countably many $x_2\in (0,1)$, $\varphi(1,\cdot)$ collapses the fiber $(0,1)\times
\{x_2\}\times \{1\}$ into a point in $(0,1)^2\times \{0\}$.
Indeed, by definition of $\varphi$ in $[0,1]^2\times (0,1]$ we have that
$$\varphi(1,(0,1)\times \{x_2\}\times \{1\})= (\chi^{(0)}(1,(0,1)\times \{x_2\}),0)$$ 
and we deduce the claim from \eqref{coll.2d} (cf. Figure 3 for an illustration of the action of $\varphi(1,\cdot)$).
\end{itemize}

\textit{\underline{Step 4:} A measure preserving map induced by $\varphi.$}
We claim that, the map $h:S\rightarrow S$ defined by
$$h(x):=\varphi(1-t^{+}(x),m(\varphi(t^{+}(x)-1,x)))$$ is well defined and measure preserving on $S,$
where
$$m(y_1,y_2,y_3):=(1-y_1,y_2,y_3)$$ and satisfies
\begin{equation}\label{h.involution}h\circ h=\operatorname{id}\quad \text{on $S.$}\end{equation}
In words the map $h$ does the following: it first sends $x$ to the set $[0,1]^2\times \{1\}$ by $\varphi(t^+(x)-1,\cdot)$ (cf. \eqref{t.plus.x.moins.1}).
It then does a reflection with respect to the set $\{x_1=1/2\}$ and then sends back the resulting point by $\varphi(1-t^+(x),\cdot)$.
First using \eqref{times.ne.dependent.pas.de.x1} and  \eqref{semi.group.varphi.time.t.pm} we get that
$$t^{+}\big(m(\varphi(t^{+}(x)-1,x))\big)= t^{+}(\varphi(t^{+}(x)-1,x))=1$$ and
$$t^{-}\big(m(\varphi(t^{+}(x)-1,x))\big)= t^{-}(\varphi(t^{+}(x)-1,x))= t^{-}(x)-t^{+}(x)+1$$ and hence
$$1-t^{+}(x)\in \left(t^{-}\big(m(\varphi(t^{+}(x)-1,x))\big),t^{+}\big(m(\varphi(t^{+}(x)-1,x))\big)\right)$$ 
implying (cf. \eqref{varphi.dans.S}) that $h(x)$ is well define and belongs to $S.$
Using again \eqref{semi.group.varphi.time.t.pm} and \eqref{times.ne.dependent.pas.de.x1} we get that
\begin{align}&t^+(h(x))= t^{+}(\varphi(1-t^{+}(x),m(\varphi(t^{+}(x)-1,x))))\nonumber\\
=&t^+(m(\varphi(t^{+}(x)-1,x)))-1+t^{+}(x)=t^+(\varphi(t^{+}(x)-1,x))-1+t^{+}(x)\nonumber\\
=&t^+(x)\label{t.hx.egal.t.x}.
\end{align}
Hence, using \eqref{semi.group.varphi.x} and \eqref{t.hx.egal.t.x}, we get, since trivially $m\circ m=\operatorname{id},$
\begin{align*}&h(h(x))=\varphi(1-t^+(x),
m(\varphi(t^+(x)-1,\varphi(1-t^+(x),m(\varphi(t^+(x)-1,x)))))\\
=&\varphi(1-t^+(x),m(m(\varphi(t^+(x)-1,x))))=\varphi(1-t^+(x),\varphi(t^+(x)-1,x))= x,
\end{align*}
showing \eqref{h.involution}.
It remains to show that $h$ is measure preserving on $S.$ 
For that, since (cf. \eqref{h.involution}) $h$ is a bijection from $S$ to $S$, 
it is enough to prove that, for every $1\leq i\leq 7,$ and every $x\in A_i,$
\begin{equation}\label{generic.form.h}h(x)=(l^{(i)}_{x_3}(x_1,x_2),x_3)
\end{equation}
for some measure preserving  map $l^{(i)}_{x_3}$ in $\mathbb{R}^2$.
\begin{itemize}
\item We first prove \eqref{generic.form.h} for $A_1.$
Recalling that, for $x \in A_1$ and $t\in [x_3-1,x_3],$ $\varphi(t,x)=(\chi^{(1-x_3)}(t,x_1,x_2),x_3-t)$,
where $\chi^{(\alpha)}(t,\cdot)$ is measure preserving in $[0,1]^2$ and that $t^{+}(x)=x_3$, we get
\begin{align*}
h(x)&=\varphi(1-x_3,m(\varphi(x_3-1,x)))=\varphi(1-x_3,m(\chi^{(1-x_3)}(x_3-1,x_1,x_2),1))\\
&=(\chi^{(0)}(1-x_3,m(\chi^{(1-x_3)}(x_3-1,x_1,x_2))),x_3),
\end{align*}
where, by abuse of notations, $m$ stands for $m(x_1,x_2)= (1-x_1,x_2)$ in second line of the previous equation.
This shows the claim.
\item
Since $a$ does not depend on $x_1$ in $A_2$ we directly get
(cf. the formula for $\varphi$ is Step 2) that
$$h(x)= m(x)= (1-x_1,x_2,x_3)\quad \text{for $x\in A_2$}$$
showing trivially the claim for $A_2$.
\item
For $x\in A_3$ since $\varphi(1-x_3,x)\in M_2\subset A_2$ we have (cf. the previous point) that $h(\varphi(1-x_3,x))=m(\varphi(1-x_3,x)).$
Hence, using \eqref{semi.group.varphi.x}, \eqref{semi.group.varphi.time.t.pm} and \eqref{times.ne.dependent.pas.de.x1},
\begin{align*}h(x)&=\varphi(1-t^{+}(x),m(\varphi(t^{+}(x)-1,x)))\\&= \varphi(x_3-1+1-x_3-1-t^{+}(x),m(\varphi(x_3-1+t^{+}(x)-1+1-x_3,x)))\\
&=\varphi(x_3-1,\varphi(1-x_3-1-t^{+}(x),m(\varphi(x_3-1+t^{+}(x)-1,\varphi(1-x_3,x)))))\\
&=\varphi(x_3-1,h(\varphi(1-x_3,x)))= \varphi(x_3-1,m(\varphi(1-x_3,x))).
\end{align*}
Hence, by definition of $\varphi$ in $A_3$ (cf. Step 2) and the fact that
 $\xi^c(\alpha,\cdot)$ and $m$ are measure preserving in $\mathbb{R}^2$, we obtain \eqref{generic.form.h} for $A_3$.
\item For $x\in A_4\cup A_5$ a simple calculation gives
$$h(x)=(1-x_1,x_2,x_3),$$ which yields trivially the claim.
\item Next for $x \in A_6$ proceeding similarly as for $A_3$ we get that
\begin{align*}h(x)&=\varphi(x_3+2,m(\varphi(-2-x_3,x)))\end{align*} and thus by definition of $\varphi$ in $A_6$ we get \eqref{generic.form.h} as for $A_3$.
\item Finally for $x\in A_7$ proceeding as for $x \in A_3$ we get that
\begin{align*}
h(x)=\varphi(1-x_3,m(\varphi(x_3-1,x)))
\end{align*}
and hence by definition of $\varphi$ in $A_7$ we deduce, as for $A_1$, the claim.
\end{itemize}

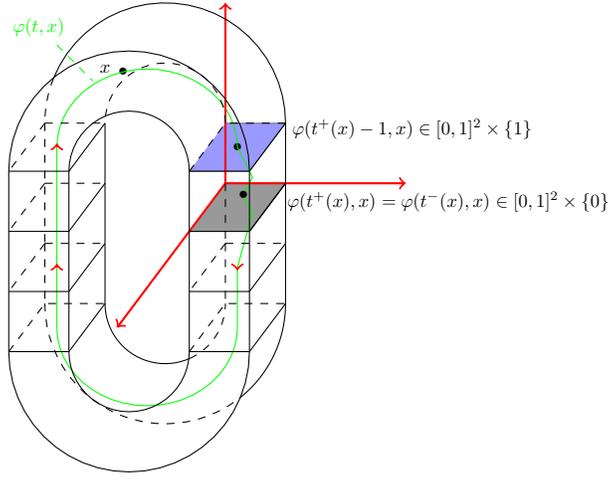
\begin{figure}
\begin{center}
\begin{tikzpicture}[scale=0.8]
\draw[fill=black, opacity=0.4] (0,0)--(1,0)--(1-3/5,-4/5)--(-3/5,-4/5);
\draw[fill=blue, opacity=0.4] (0,1)--(1,1)--(1-3/5,1-4/5)--(-3/5,1-4/5);
\draw (0.5-3/10,1-4/10) node[scale=0.7]{$\bullet$};
\draw (3.4-3/10,1.3-4/10) node[scale=0.7]{$\varphi(t^{+}(x)-1,x)\in [0,1]^2\times \{1\}$};
\draw (-2.0,1.9) node[scale=0.7]{$x$};\draw (-1.7,1.85) node[scale=0.7]{$\bullet$};

\draw [domain=0:180,green] plot ({-3/10-1+1.5*cos(\x)}, {1+-4/10+ 1.3*sin(\x)});
\draw [domain=180:360,green] plot ({-3/10-1+1.5*cos(\x)}, {-2+-4/10+ 1.3*sin(\x)});
\draw[green] (-3/10-2.5,1-4/10)--(-3/10-2.5,-2-4/10);
\draw[green](-3/10+0.5,1-4/10)--(-3/10+0.75,0.5-4/10)--(-3/10+0.6,0.2-4/10)--
(-3/10+0.75,-0.1-4/10)--(0.5-3/10,-1-4/10)--(0.5-3/10,-2-4/10);
\draw[arrows=->,thick,red](-3/10-2.5,-4/10+1)--(-3/10-2.5,-4/10+1);
\draw[arrows=->,thick,red](-3/10-2.5,-4/10-1)--(-3/10-2.5,-4/10-1);
\draw[arrows=->,thick,red](-3/10+0.5,-4/10-1)--(-3/10+0.5,-4/10-1.05);

\draw[green] (-3/10-2.8,3-4/10) node[scale=0.7]{$\varphi(t,x)$};
\draw[green,dashed](-3/10-2.8+0.3,3-4/10-0.3)--(-3/10-2.8+0.9,3-4/10-0.9);
\draw (-3/10+0.6,0.2-4/10) node[scale=0.7]{$\bullet$};\draw (-3/10+4,0.1-4/10)
node[scale=0.7]{$\varphi(t^+(x),x)=\varphi(t^-(x),x)\in [0,1]^2\times\{0\}$};
\draw[arrows=->,thick,red](0,0)--(0,3);\draw[arrows=->,thick,red](0,0)--(3,0);
\draw[arrows=->,thick,red](0,0)--(-3*3/5,-3*4/5);
\draw[dashed](0,0)--(0,-1);
\draw (-3/5,-4/5-1)--(-3/5,-4/5+1);\draw (1-3/5,-4/5-1)--(1-3/5,-4/5+1);
\draw (1,-1)--(1,1); \draw[dashed] (1,1)--(0,1);
\draw (1,1)--(1-3/5,-4/5+1);\draw (1,0)--(1-3/5,-4/5);
\draw (1,0)--(1-3/5,-4/5);\draw (1,-1)--(1-3/5,-4/5-1);
\draw (-3/5,-4/5)--(-3/5+1,-4/5);\draw (-3/5,-4/5+1)--(-3/5+1,-4/5+1);
\draw [dashed](-3/5,-4/5+1)--(0,1); \draw [dashed](0,-1)--(1,-1);
\draw[dashed] (0,-1)--(-3/5,-4/5-1);
\draw(-3/5,-4/5-1)--(-3/5+1,-4/5-1);
\draw [domain=0:180,dashed] plot ({-1+cos(\x)}, {1+sin(\x)});
\draw [domain=0:158] plot ({-1+2*cos(\x)}, {1+2*sin(\x)});
\draw [domain=158:180,dashed] plot ({-1+2*cos(\x)}, {1+2*sin(\x)});
\draw [domain=0:180] plot ({-1-3/5+cos(\x)}, {1-4/5+sin(\x)});
\draw [domain=0:180] plot ({-1-3/5+2*cos(\x)}, {1-4/5+2*sin(\x)});
\draw [domain=180:294] plot ({-1+cos(\x)}, {-2+sin(\x)});
\draw [domain=294:360,dashed] plot ({-1+cos(\x)}, {-2+sin(\x)});
\draw [domain=180:308,dashed] plot ({-1+2*cos(\x)}, {-2+2*sin(\x)});
\draw [domain=308:360] plot ({-1+2*cos(\x)}, {-2+2*sin(\x)});
\draw [domain=180:360] plot ({-1-3/5+cos(\x)}, {-2-4/5+sin(\x)});
\draw [domain=180:360] plot ({-1-3/5+2*cos(\x)}, {-2-4/5+2*sin(\x)});
\draw[dashed] (-3,1)--(-2,1);
\draw [dashed](-3,-1)--(-2,-1);
\draw[dashed] (-3,1)--(-3,-1);\draw (-2,1)--(-2,-1);
\draw [dashed](-3,1)--(-3-3/5,1-4/5);\draw (-2,1)--(-2-3/5,1-4/5);
\draw [dashed](-3,0)--(-3-3/5,0-4/5);\draw (-2,0)--(-2-3/5,-4/5);
\draw [dashed](-3,-1)--(-3-3/5,-1-4/5);\draw (-2,-1)--(-2-3/5,-1-4/5);
\draw [dashed](-3,0)--(-2,0);
\draw (-3-3/5,-4/5)--(-2-3/5,-4/5);
\draw (-3-3/5,1-4/5)--(-2-3/5,1-4/5);
\draw (-3-3/5,-1-4/5)--(-2-3/5,-1-4/5);
\draw (-3-3/5,1-4/5)--(-3-3/5,-1-4/5);\draw (-2-3/5,1-4/5)--(-2-3/5,-1-4/5);
\draw [dashed] (0-3/5,-2-4/5)--(0,-2)--(1,-2);
\draw (0-3/5,-2-4/5)--(1-3/5,-2-4/5)--(1,-2);
\draw (0-3/5,-2-4/5)--(0-3/5,-1-4/5);\draw (1-3/5,-2-4/5)--(1-3/5,-1-4/5);\draw (1,-2)--(1,-1);
\draw [dashed](0,-1)--(0,-2);
\draw [dashed] (-3-3/5,-2-4/5)--(-3,-2)--(-2,-2);
\draw (-3-3/5,-2-4/5)--(-2-3/5,-2-4/5)--(-2,-2);
\draw (-3-3/5,-2-4/5)--(-3-3/5,-1-4/5);\draw (-2-3/5,-2-4/5)--(-2-3/5,-1-4/5);\draw (-2,-2)--(-2,-1);
\draw [dashed](-3,-1)--(-3,-2);
\end{tikzpicture}
\caption{Definition of $t^{\pm}(x)$.}\end{center}
\end{figure}

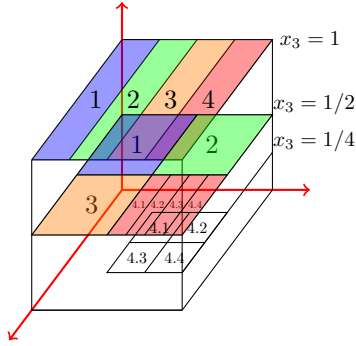
\begin{figure}
\begin{center}

\begin{tikzpicture}[scale=0.5]
\draw[arrows=->,thick,red](0,0)--(0,5);\draw[arrows=->,thick,red](0,0)--(5,0);
\draw[arrows=->,thick,red](0,0)--(-3,-4);
\draw (-4*3/5,-4*4/5)--(4+-4*3/5,-4*4/5)--(4,0)--(4,4)--(0,4);
\draw (1,4)--(1-4*3/5,4-4*4/5);
\draw (2,4)--(2-4*3/5,4-4*4/5);
\draw (3,4)--(3-4*3/5,4-4*4/5);
\draw[fill=blue, opacity=0.4] (0,4)--(1,4)--(1-4*3/5,4-4*4/5)--(-4*3/5,4-4*4/5);
\draw[fill=green, opacity=0.4] (1,4)--(2,4)--(2-4*3/5,4-4*4/5)--(1-4*3/5,4-4*4/5);
\draw[fill=orange, opacity=0.4] (2,4)--(3,4)--(3-4*3/5,4-4*4/5)--(2-4*3/5,4-4*4/5);
\draw[fill=red, opacity=0.4] (3,4)--(4,4)--(4-4*3/5,4-4*4/5)--(3-4*3/5,4-4*4/5);
\draw (0.5-2*3/5,4-2*4/5) node{$1$};\draw (1.5-2*3/5,4-2*4/5) node{$2$};
\draw (2.5-2*3/5,4-2*4/5) node{$3$};\draw (3.5-2*3/5,4-2*4/5) node{$4$};
\draw (1-3/5,2-4/5) node{$1$};\draw (3-3/5,2-4/5) node{$2$};
\draw (1-3*3/5,2-3*4/5) node{$3$};
\draw (-0.75+3-3*3/5,2-3*4/5) node[scale=0.4]{$4.1$};
\draw (-0.25+3-3*3/5,2-3*4/5) node[scale=0.4]{$4.2$};
\draw (0.25+3-3*3/5,2-3*4/5) node[scale=0.4]{$4.3$};
\draw (0.75+3-3*3/5,2-3*4/5) node[scale=0.4]{$4.4$};

\draw (-4*3/5,-4*4/5)--(4+-4*3/5,-4*4/5);
\draw (4+-4*3/5,-4*4/5)--(4+-4*3/5,4+-4*4/5);
\draw (-4*3/5,4-4*4/5)--(4+-4*3/5,4-4*4/5);
\draw (4+-4*3/5,4+-4*4/5)--(4,4);
\draw (-4*3/5,-4*4/5)--(-4*3/5,4+-4*4/5);
\draw(-4*3/5,4+-4*4/5)--(0,4);
\draw[fill=blue, opacity=0.4] (-2*3/5,2+-2*4/5)--(0,2)--(2,2)--(2-2*3/5,2+-2*4/5)--(-2*3/5,2+-2*4/5);
\draw[fill=green, opacity=0.4] (2-2*3/5,2+-2*4/5)--(2,2)--(4,2)--(4-2*3/5,2+-2*4/5)--(2-2*3/5,2+-2*4/5);
\draw[fill=red, opacity=0.4]
(2-4*3/5,2+-4*4/5)--(2-2*3/5,2-2*4/5)--(4-2*3/5,2-2*4/5)--(4-4*3/5,2-4*4/5)--(2-4*3/5,2+-4*4/5);
\draw[fill=orange, opacity=0.4]
(-4*3/5,2-4*4/5)--(-2*3/5,2-2*4/5)--(2-2*3/5,2-2*4/5)--(2-4*3/5,2-4*4/5)--(-4*3/5,2-4*4/5);
\draw (0,2)--(4,2)--(4-4*3/5,2-4*4/5)--(-4*3/5,2-4*4/5)--(0,2);
\draw (2,2)--(2-4*3/5,2-4*4/5);\draw (-2*3/5,2-2*4/5)--(4-2*3/5,2-2*4/5);
\draw (3.5-4*3/5,2-4*4/5)--(3.5-2*3/5,2-2*4/5);
\draw (3-4*3/5,2-4*4/5)--(3-2*3/5,2-2*4/5);
\draw (2.5-4*3/5,2-4*4/5)--(2.5-2*3/5,2-2*4/5);
\draw (4-4*3/5,1-4*4/5)--(2-4*3/5,1-4*4/5)--(2-2*3/5,1-2*4/5)--(4-2*3/5,1-2*4/5)--(4-4*3/5,1-4*4/5);
\draw (4-4*3/5,1-4*4/5)--(4,1);
\draw (5.1,1.3) node[scale=0.8]{$x_3=1/4$};
\draw (5.1,2.3) node[scale=0.8]{$x_3=1/2$};
\draw (5,4) node[scale=0.8]{$x_3=1$};
\draw (4-3*3/5,1-3*4/5)--(2-3*3/5,1-3*4/5);
\draw (3-4*3/5,1-4*4/5)--(3-2*3/5,1-2*4/5);
\draw (3.5-3.5*3/5,1-3.5*4/5) node[scale=0.6]{4.4};
\draw (2.5-3.5*3/5,1-3.5*4/5) node[scale=0.6]{4.3};
\draw (3.5-2.5*3/5,1-2.5*4/5) node[scale=0.6]{4.2};
\draw (2.5-2.5*3/5,1-2.5*4/5) node[scale=0.6]{4.1};
\end{tikzpicture}
\caption{The action of the flow $\varphi$ generated by $a$: the image of every rectangle $i$ (at $x_3=1$) is sent
by $\varphi(1/2,\cdot)$ to the corresponding square (at $x_3=1/2$). Similarly, $\varphi(1/2,\cdot)$ sends in
particular every rectangle $4.i$ at its corresponding square at height $1/4$.}\end{center}
\end{figure}

\textit{\underline{Step 5:} construction of two distinct flows for $a$.}
With the help of $\varphi$ we now construct two measure preserving distinct flows $\phi$ and $\psi$ of
$a$
where we recall that, for every $x\in S$, $\varphi(\cdot,x)$ is a measure preserving flow of $a$ defined on $[t^{-}(x),t^{+}(x)]$.
Using crucially the collapsing of fibers discussed in Step 3 we will show how to extend $\varphi(\cdot,x)$ outside
$[t^{-}(x),t^{+}(x)]$ in two distinct ways.
Let (see \eqref{period.indep.of.x}), for every $x\in S,$
$$T:=6+3\pi= t^+(x)-t^-(x),$$ which can be seen as the period of the orbit $\varphi(\cdot,x)$ recalling (cf.
\eqref{varphi.extreme.equale})
$$\varphi(t^{+}(x),x)=\varphi(t^{-}(x),x).$$
We first define $\phi$ by "periodicity":
$$ \phi(t,x):= \left\{\begin{array}{cl}
x &\text{for $t\in \mathbb{R}$ and $x\in \mathbb{R}^3\setminus S$}\\
 \varphi(t-kT,x) & \text{for $t\in\mathbb{R}$ and $x\in S$}
  \end{array}
   \right.
$$
where $k\in \mathbb{Z}$ is the unique integer such that
$$t-kT\in (t^-(x),t^+(x)].$$

The definition of $\psi$ is more involved. First we define the set $W\subset S$ by
$$W:=\{x\in S:\varphi(t^{+}(x)-1,x)\in (0,1)\times Z\times \{1\}\}.$$
Equivalently $W$ is the set of points $x$ in $S$ whose orbit $\varphi([t^-(x),t^{+}(x)],x)$ goes throw the set
$(0,1)\times Z\times \{1\}$.
Since $Z$ is countable and $\varphi$ is measure preserving we get that
$|W|=0.$
Next for every $x\in S\setminus W$ we claim that
\begin{equation}\label{collapsion.avec.tilde.x}
\varphi(t^{+}(x),x)= \varphi(t^{+}(x),h(x)),
\end{equation} where $h$ is the measure preserving map defined in Step 4.
Indeed, using \eqref{semi.group.varphi.x}, \eqref{t.plus.x.moins.1} and \eqref{coll.2d},
\begin{align*} \varphi(t^{+}(x),h(x))&=\varphi(t^{+}(x),\varphi(1-t^{+}(x),m(\varphi(t^{+}(x)-1,x))))\\&=
\varphi(1,m(\varphi(t^{+}(x)-1,x)))=\varphi(1,\varphi(t^{+}(x)-1,x))=\varphi(t^{+}(x),x).\end{align*}
We now define $\psi$ as follows:
$$ \psi(t,x):= \left\{\begin{array}{cl}
x &\text{for $t\in \mathbb{R}$ and $x\in \mathbb{R}^3\setminus S$}\\
\varphi(t-kT,x) & \text{for $t\in\mathbb{R}$ and $x\in W$}\\
\varphi(t,x) &\text{for $t\in [t^-(x),t^+(x)]$ and $x\in S\setminus W$}\\
\varphi(t-kT,h(x)) &\text{for $t\notin [t^-(x),t^+(x)]$ and $x\in S\setminus W$}
  \end{array}
   \right.
$$ where
as before, $k\in \mathbb{Z}$ is the unique integer such that
$$t-kT\in (t^-(x),t^+(x)].$$
Note that the previous definition makes sense since (cf. \eqref{t.hx.egal.t.x})
$$t^{\pm}\left(h(x)\right)=t^{\pm}(x).$$

See Figure 4  for an illustration of the orbits of $\phi$ and $\psi$: The green closed
curbed represents the image of $$\{\phi(t,x):t\in [t^{-}(x)-kT,t^{+}(x)-kT]\}\quad\text{for any $k\in
\mathbb{Z}$}$$
where $T=t^{+}(x)-t^{-}(x)$ is the "period" of the curb; the orange closed curbed represents the image of
$$\{\psi(t,x):t\in [t^{-}(x)-kT,t^{+}(x)-kT]\}\quad \text{for any $k\in \mathbb{Z}\setminus\{0\}$}$$ (for
$k=0$ it coincides with the green curbed).\smallskip

\begin{figure}
\begin{center}
\begin{tikzpicture}[scale=0.8]
\draw[fill=black, opacity=0.4] (0,0)--(1,0)--(1-3/5,-4/5)--(-3/5,-4/5);
\draw[fill=blue, opacity=0.4] (0,1)--(1,1)--(1-3/5,1-4/5)--(-3/5,1-4/5);
\draw[green,align=left] (-3/10+3.5,2.7-4/10) node[scale=0.7]{$\varphi(t^+(x)-1,x)$};
\draw[YellowOrange,align=left] (-3/10+3.7,1.7-4/10) node[scale=0.7]{$m(\varphi(t^+(x)-1,x))$};
\draw (-3/10+0.6,0.2-4/10) node[scale=0.7]{$\bullet$};

\draw (-0.8*3/5-2.5+3,1-0.8*4/5) node[scale=0.8]{$\bullet$};
\draw (-0.2*3/5-2.6+3,1-0.2*4/5) node[scale=0.8] {$\bullet$};
\draw[dashed](-0.8*3/5-2.5+3,1-0.8*4/5)--(-3/10+3.5,1.5-4/10);
\draw[dashed](-0.2*3/5-2.6+3,1-0.2*4/5)--(-3/10+3.5,2.5-4/10);
\draw (-2.0,2.05) node[scale=0.7]{$x$};\draw (-1.7,2) node[scale=0.7]{$\bullet$};
\draw (-1.6,1.45) node[scale=0.7]{$h(x)$};\draw (-2,1.5) node[scale=0.7]{$\bullet$};
\draw [domain=5:180,green,thick] plot ({-3/10-2.42+1.5+1.5*cos(\x)},{1.15-4/10+ 1.3*sin(\x)});
\draw [domain=-10:180,YellowOrange,thick] plot ({-3/10-0.5*3/5-2.35+1.5+1.5*cos(\x)},{1.15-0.5*4/5-4/10+
1.3*sin(\x)});
\draw [domain=180:360,green,thick] plot ({-3/10-2.42+1.5+1.5*cos(\x)},{-3+1.15-4/10+ 1.3*sin(\x)});
\draw [domain=183:360,YellowOrange,thick] plot ({-3/10-0.5*3/5-2.34+1.5+1.5*cos(\x)},{-3+1.15-0.5*4/5-4/10+
1.3*sin(\x)});
\draw[green,thick]({-3/10-2.42+1.5+1.5*cos(360)},{-3+1.15-4/10+ 1.3*sin(360)})--({-3/10-2.42+1.5+1.5*cos(360)},{1-3+1.15-4/10+ 1.3*sin(360)});
\draw[YellowOrange,thick]({-3/10-0.5*3/5-2.34+1.5+1.5*cos(360)},{-3+1.15-0.5*4/5-4/10+
1.3*sin(360)})--({-3/10-0.5*3/5-2.34+1.5+1.5*cos(360)},{1-3+1.15-0.5*4/5-4/10+
1.3*sin(360)});
\draw[green,thick]({-3/10-2.42+1.5+1.5*cos(5)},{1.15-4/10+ 1.3*sin(5)})--(-3/10+0.6,0.2-4/10);
\draw[YellowOrange,thick]({-3/10-0.5*3/5-2.35+1.5+1.5*cos(2)},{1.15-0.5*4/5-4/10+
1.3*sin(2)})--(-3/10+0.6,0.2-4/10);
\draw[green,thick](-3/10-2.42+1.5+1.5,1.15-4/10 -1)--(-3/10+0.6,0.2-4/10);
\draw[YellowOrange,thick]({-3/10-0.5*3/5-2.35+1.5+1.5*cos(360)},{-2+1.15-0.5*4/5-4/10+
1.3*sin(360)})--(-3/10+0.6,0.2-4/10);
\draw[green,thick]({-3/10-2.42+1.5+1.5*cos(360)},{-2+1.15-4/10+ 1.3*sin(360)})-- (-3/10+0.6,0.2-4/10);
\draw[green,thick] (-3/10-2.42,1.15-4/10)--(-3/10-2.42,-1-0.85-4/10);
\draw[YellowOrange,thick] (-3/10-0.4*3/5-2.39,-0.4*4/5+1.1-4/10)--(-0.4*3/5-3/10-2.39,-1-0.4*4/5-1-4/10);

\draw[green] (-3/10-2.8,3-4/10) node[scale=0.7]{$\phi(t,x)$};
\draw[green,dashed](-3/10-2.8+0.3,3-4/10-0.3)--(-3/10-2.8+0.9,3-4/10-0.9);
\draw[YellowOrange] (-0.3+1.5,1.2+1.5) node[scale=0.7]{$\psi(t,x)$};
\draw[YellowOrange,dashed](-0.3,1.2)--(-0.3+1.2,1.2+1.2);

\draw[arrows=->,thick,red](0,0)--(0,3);\draw[arrows=->,thick,red](0,0)--(3,0);
\draw[arrows=->,thick,red](0,0)--(-3*3/5,-3*4/5);
\draw[dashed](0,0)--(0,-1);
\draw (-3/5,-4/5-1)--(-3/5,-4/5+1);\draw (1-3/5,-4/5-1)--(1-3/5,-4/5+1);
\draw (1,-1)--(1,1); \draw[dashed] (1,1)--(0,1);
\draw (1,1)--(1-3/5,-4/5+1);\draw (1,0)--(1-3/5,-4/5);
\draw (1,0)--(1-3/5,-4/5);\draw (1,-1)--(1-3/5,-4/5-1);
\draw (-3/5,-4/5)--(-3/5+1,-4/5);\draw (-3/5,-4/5+1)--(-3/5+1,-4/5+1);
\draw [dashed](-3/5,-4/5+1)--(0,1); \draw [dashed](0,-1)--(1,-1);
\draw[dashed] (0,-1)--(-3/5,-4/5-1);
\draw(-3/5,-4/5-1)--(-3/5+1,-4/5-1);
\draw [domain=0:180,dashed] plot ({-1+cos(\x)}, {1+sin(\x)});
\draw [domain=0:158] plot ({-1+2*cos(\x)}, {1+2*sin(\x)});
\draw [domain=158:180,dashed] plot ({-1+2*cos(\x)}, {1+2*sin(\x)});
\draw [domain=0:180] plot ({-1-3/5+cos(\x)}, {1-4/5+sin(\x)});
\draw [domain=0:180] plot ({-1-3/5+2*cos(\x)}, {1-4/5+2*sin(\x)});
\draw [domain=180:294] plot ({-1+cos(\x)}, {-2+sin(\x)});
\draw [domain=294:360,dashed] plot ({-1+cos(\x)}, {-2+sin(\x)});
\draw [domain=180:308,dashed] plot ({-1+2*cos(\x)}, {-2+2*sin(\x)});
\draw [domain=308:360] plot ({-1+2*cos(\x)}, {-2+2*sin(\x)});
\draw [domain=180:360] plot ({-1-3/5+cos(\x)}, {-2-4/5+sin(\x)});
\draw [domain=180:360] plot ({-1-3/5+2*cos(\x)}, {-2-4/5+2*sin(\x)});
\draw[dashed] (-3,1)--(-2,1);
\draw [dashed](-3,-1)--(-2,-1);
\draw[dashed] (-3,1)--(-3,-1);\draw (-2,1)--(-2,-1);
\draw [dashed](-3,1)--(-3-3/5,1-4/5);\draw (-2,1)--(-2-3/5,1-4/5);
\draw [dashed](-3,0)--(-3-3/5,0-4/5);\draw (-2,0)--(-2-3/5,-4/5);
\draw [dashed](-3,-1)--(-3-3/5,-1-4/5);\draw (-2,-1)--(-2-3/5,-1-4/5);
\draw [dashed](-3,0)--(-2,0);
\draw (-3-3/5,-4/5)--(-2-3/5,-4/5);
\draw (-3-3/5,1-4/5)--(-2-3/5,1-4/5);
\draw (-3-3/5,-1-4/5)--(-2-3/5,-1-4/5);
\draw (-3-3/5,1-4/5)--(-3-3/5,-1-4/5);\draw (-2-3/5,1-4/5)--(-2-3/5,-1-4/5);

\draw (-3-3/5,1-4/5)--(-3-3/5,-1-4/5);\draw (-2-3/5,1-4/5)--(-2-3/5,-1-4/5);
\draw [dashed] (0-3/5,-2-4/5)--(0,-2)--(1,-2);
\draw (0-3/5,-2-4/5)--(1-3/5,-2-4/5)--(1,-2);
\draw (0-3/5,-2-4/5)--(0-3/5,-1-4/5);\draw (1-3/5,-2-4/5)--(1-3/5,-1-4/5);\draw (1,-2)--(1,-1);
\draw [dashed](0,-1)--(0,-2);
\draw [dashed] (-3-3/5,-2-4/5)--(-3,-2)--(-2,-2);
\draw (-3-3/5,-2-4/5)--(-2-3/5,-2-4/5)--(-2,-2);
\draw (-3-3/5,-2-4/5)--(-3-3/5,-1-4/5);\draw (-2-3/5,-2-4/5)--(-2-3/5,-1-4/5);\draw (-2,-2)--(-2,-1);
\draw [dashed](-3,-1)--(-3,-2);
\end{tikzpicture}
\caption{The two distinct flows $\phi$ and $\psi$ starting at a point $x\in W$.
}\end{center}
\end{figure}
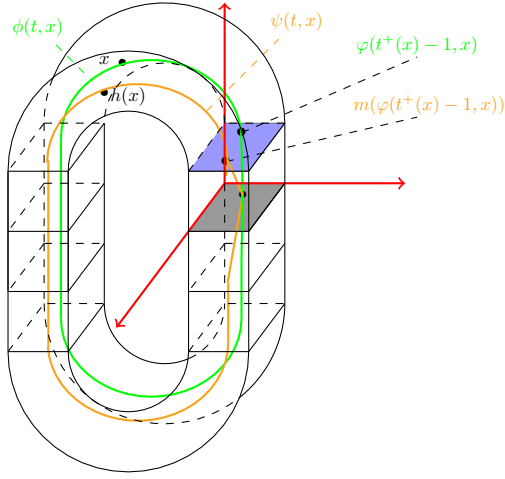

\textit{\underline{Step 6:} properties of $\phi$ are $\psi$.}\\
\textit{\underline{Step 6.1:} $\phi$ and $\psi$ are flows of $a.$}
First from \eqref{varphi.extreme.equale} we deduce that for every $x\in S$ the map $\phi(\cdot,x)$ is continuous
in $\mathbb{R}.$
Hence, recalling \eqref{varphi.flot.on.S}, we directly  get that
\begin{equation}\label{flow.phi}
\phi(t,x)=x+\int_0^t a(\phi(s,x))ds \quad \text{for every $x\in \mathbb{R}^3$ and $t\in \mathbb{R}.$}
\end{equation}
Similarly using \eqref{varphi.extreme.equale}, \eqref{semi.group.varphi.time.t.pm}, \eqref{semi.group.varphi.x}
and \eqref{collapsion.avec.tilde.x} we get that, $\psi(\cdot,x)$ is continuous in $\mathbb{R}.$
Hence, again by \eqref{varphi.flot.on.S},
we deduce  that
\begin{equation}\label{flow.psi}
\psi(t,x)=x+\int_0^t a(\psi(s,x))ds\quad \text{for every $x\in \mathbb{R}^3$ and $t\in \mathbb{R}.$}
\end{equation}

\textit{\underline{Step 6.2:} $\phi$ and $\psi$ satisfy the group property a.e. in $\mathbb{R}^3$.} Using \eqref{semi.group.varphi.x} and the definition of $\phi$ and $\varphi$ we easily get that for every $x\in S$, for every $t_1\in \mathbb{R}$ and for every
$t_2\in \mathbb{R}\setminus \{t^+(x)+nT:n\in \mathbb{Z}\},$
\begin{equation}\label{phi.psi.semi.group}\phi(t_1,\phi(t_2,x))=\phi(t_1+t_2,x)\quad \text{and}\quad \psi(t_1,\psi(t_2,x))=\psi(t_1+t_2,x).\end{equation}
Obviously \eqref{phi.psi.semi.group} is satisfied for every $x\notin S\cup ([0,1]^2\times\{0\})$ and every $t_1,t_2\in \mathbb{R}$ since in that case $\phi(\cdot,x)=\psi(\cdot,x)=x.$
At the end we have showed that for a.e. $x\in \mathbb{R}^3$, for every $t_1\in \mathbb{R}$ and for every $t_2\in \mathbb{R}$ except an at most countable set (depending of $x$) \eqref{phi.psi.semi.group} is satisfied.

\textit{\underline{Step 6.3:} $\phi(t,\cdot)$ and $\psi(t,\cdot)$ are bijections a.e. in $\mathbb{R}^3$.}
For every $t\neq 0$ define
$$Q_t:=\{x\in S|\text{ }t^{+}(x)\in t+\mathbb{Z}T\}.$$
and $$Q_0:=[0,1]^2\times \{0\}.$$
From \eqref{t.plus.moins.1.measure}, we directly get that
$|Q_t|=0.$
As a direct consequence of the group property established in Step 6.2, we deduce that, for every $t\in \mathbb{R}$,
$\phi(t,\cdot),\psi(t,\cdot)$ are both bijections from $\mathbb{R}^3\setminus (Q_{t}\cup Q_0)$ onto
$\mathbb{R}^3\setminus (Q_{-t}\cup Q_0)$ with
$$\phi(t,\cdot)^{-1}=\phi(-t,\cdot)\quad \text{and} \quad \psi(t,\cdot)^{-1}=\psi(-t,\cdot)\quad \text{in
$\mathbb{R}^3\setminus (Q_{-t}\cup Q_0)$}.$$
Note that using \eqref{varphi.dans.S}, \eqref{varphi.dans.plan.critique} and \eqref{collapsion.avec.tilde.x} we get from the definition of $\phi$ and $\psi$ that
$$Q_t=\{x\in S|\text{ }\phi(t,x)\in [0,1]^2\times \{0\}\}=\{x\in S|\text{ }0\psi(t,x)\in [0,1]^2\times \{0\}\}.
$$
\smallskip

\textit{\underline{Step 6.4:} $\phi(t,\cdot)$ and $\psi(t,\cdot)$  preserve the Lebesgue measure.}
We claim that for every $t\in \mathbb{R}$, $\phi(t,\cdot)$ and $\psi(t,\cdot)$ both preserve the Lebesgue measure
in $\mathbb{R}^3$.
We start with $\phi.$ As $\phi$ is a
bijection (cf. Step 6.3) from  $\mathbb{R}^3\setminus (Q_t\cup Q_0)$ onto $\mathbb{R}^3\setminus (Q_{-t}\cup Q_0)$ and
$|Q_t\cup Q_0|=0$, it is enough to show that for any $x\in \mathbb{R}^3\setminus (Q_t\cup Q_0)$ there exists a neighbourhood $U$ of $x$ such that $\phi(t,\cdot)|_U:U\rightarrow \phi(t,U)$  preserves the measure. We can assume that $x\in S$ otherwise the claim is trivial since $\phi(t,\cdot)$ is the identity on $\mathbb{R}^3\setminus (S\cup Q_t\cup Q_0)$. Then since $x\notin Q_t$ we have by definition that $t\notin t^{+}(x)+\mathbb{Z}T$. Hence by continuity of $t^{\pm}$ (cf. \eqref{times.ne.dependent.pas.de.x1}) there exist a neighbourhood $U$ of $x$ in $S$ and $k\in \mathbb{Z}$ such that
$$t-kT\in (t^{-}(y),t^{+}(y))\quad \text{for every $y$ in $U.$}$$
Since then by definition of $\phi$ we have, for every $y\in U$,
$$\phi(t,y)=\varphi(t-kT,y)$$ we conclude by \eqref{varphi.preserve.measure} that $\phi(t,\cdot)|_U:U\rightarrow \phi(t,U)$  preserves the measure.

We now deal with $\psi.$ Exactly as before it is enough to prove, for any $x\in S\setminus (Q_t\cup Q_0)$, the existence of the neighbourhood $U$ of $x$ in $S$ such that $\psi(t,\cdot)|_U:U\rightarrow \psi(t,U)$  preserves the measure.
Again exactly as before we can find a neighbourhood $U$ of $x$ in $S$ and $k\in \mathbb{Z}$ such that
$$t-kT\in (t^{-}(y),t^{+}(y))\quad \text{for every $y$ in $U.$}$$
If $k=0$ we are done using \eqref{varphi.preserve.measure} since then, by definition of $\psi$, for every $y\in U,$
$$\psi(t,y)=\varphi(t,y).$$
We can therefore assume that $k\neq 0.$ In that case, by definition of $\psi$, we have, for every $y\in U\setminus W$
$$\psi(t,y)=\varphi(t-kT,h(y)).$$
Since $h$ and $\varphi(t-kT,\cdot)$ are measure preserving we get that, using \eqref{t.hx.egal.t.x}, the map
$y\rightarrow \varphi(t-kT,h(y))$ preserves the measure in $U.$
Since $\psi$ and $y\rightarrow \varphi(t-kT,h(y))$ only differ on the null set $W$ we  get that $\psi(t,\cdot)$ preserves as well the measure in $U.$\smallskip

\textit{\underline{Step 6.5:} $\phi$ and $\psi$ differ on a set of positive Lebesgue measure.}
By definition of $\varphi$ in $A_2$ (cf. Step 2) we easily see that, for every $x\in A_2\cap \{x_2\leq -1\}$ and $t\in [0,\pi]$,
$$\varphi(t,x)\in A_2\quad \text{and the first component of $\varphi(t,x)$ is simply $x_1.$}$$
Moreover for every $x\in A_2$ recall that (cf. Step 4) $h(x)=(1-x_1,x_2,x_3).$
Hence, by definition, for every $x\in (A_2\cap \{x_2\leq -1\})\setminus W$ and $t\in [T,T+\pi]$, as $t-T\in (t^-(x),t^+(x)]$ we have
$$\phi(t,x)=\varphi(t-T,x)\quad \text{and}\quad \psi(t,x)=\varphi(t-T,h(x))$$
and therefore
\begin{equation}\label{psi.not.egal.phi}
\phi^1(t,x)=x_1\quad \text{and}\quad 1-x_1= \psi^1(t,x).
 \end{equation}
Since $|W|=0$ the previous equation shows in particular that $\phi$ and $\psi$ differ on a set with positive
Lebesgue measure in $\mathbb{R}^4$.\smallskip

Combining Steps 6.1, 6.2 and 6.3 and 6.5 we have proved the existence of two distinct measure preserving flows of $a$ satisfying the group property a.e..\smallskip

\textit{\underline{Step 7:} Non uniqueness for the transport equation.}
Let $u_0\in C^{\infty}_c(\mathbb{R}^3)$.
We claim that $v,w\in L^{\infty}([0,\infty)\times \mathbb{R}^3)$ defined by
$$v(t,x):= u_0(\phi(-t,x))\quad \text{and}\quad w(t,x):= u_0(\psi(-t,x))$$ both solve
$$\partial_tu+\langle b; \nabla_x u\rangle=0\quad \text{and}\quad u(0,\cdot)=u_0(\cdot),$$
in the weak sense.
We will only prove it for $v$ the proof for $w$ being exactly identical.
We have to prove that for every $h\in C^{\infty}_c([0,\infty)\times \mathbb{R}^3)$
\begin{align*}
\int_0^{\infty}\int_{\mathbb{R}^3}-v(t,x)\left(\partial_th(t,x)-\langle a;\nabla_x h(t,x)\rangle\right)dxdt
=\int_{\mathbb{R}^3} u_0(x)h(0,x)dx.
\end{align*}
Now since $\phi(t,\cdot)$  preserves the
Lebesgue measure (cf. Step 6.4)
and since (cf. \eqref{flow.phi}), for a.e $x\in \mathbb{R}^3$ the map
$$t\rightarrow \phi(t,x)$$ is Lipschitz on $\mathbb{R}$ with derivative $a(\phi(t,x))$, we get, for $h\in C^{\infty}_c([0,\infty)\times \mathbb{R}^3),$
\begin{align*}
&\int_0^{\infty}\int_{\mathbb{R}^3}-v(t,x)\left(\partial_t h(t,x)
+\langle a;\nabla_x h(t,x)\rangle\right)dxdt\\
=&\int_0^{\infty}\int_{\mathbb{R}^3}-u_0(x)\left(\partial_t h(t,(\phi(t,x)))+
\langle a(\phi(t,x)),\nabla_x h(t,\phi(t,x))\rangle\right)dxdt\\
=&\int_{\mathbb{R}^3}\int_{0}^{\infty}-u_0(x)\left(\partial_t h(t,(\phi(t,x)))+
\langle a(\phi(t,x));\nabla_x h(t,\phi(t,x))\rangle\right)dtdx\\
=&\int_{\mathbb{R}^3}\int_{0}^{\infty}-u_0(x)\frac{d}{dt}\left(h(t,\phi(t,x))\right)dtdx\\
=&\int_{\mathbb{R}^3}u_0(x)h(0,x)dx
\end{align*}
which proves the claim.
Finally choose $u_0$ as a smooth function with compact support such that $u_0(x)=x_1$ in $S.$
Then using \eqref{psi.not.egal.phi} we get that $v-w$ is  different from $0$ on a set of positive Lebesgue measure
set and solves \eqref{theorem.supp.comp.R.3}, which proves the second part of the theorem and concludes the proof.\smallskip

\textit{\underline{Step 8:} Non existence of a flow.}
First we define our vector field $\tilde{a}$ as follows:
$$\tilde{a}=\left\{\begin{array}{cl}
a&\text{in $\mathbb{R}^3\setminus A_7$}\\
(0,0,-1)&\text{in $A_7$}.
\end{array}
\right.
$$
Proceeding as in Step 1.2, we see that $\tilde{a}$ is measurable bounded, has compact support and is divergence free in $\mathbb{R}^3.$ Moreover it is piecewise smooth in $\mathbb{R}^3\setminus \{[0,1]^2\times \{0\}\}$
We now establish that no map $\varphi:\mathbb{R}\times \mathbb{R}^3 \rightarrow \mathbb{R}^3$ satisfies
\begin{equation}\label{flow.a.tilde}\varphi(t,x)=x+\int_0^t\tilde{a}(\varphi(s,x))ds\quad \text{for a.e. $x\in \mathbb{R}^3$ and for every $t\in \mathbb{R}$},\end{equation}
satisfies the group property a.e. and is such that
$$\varphi(t,\cdot)\quad \text{is measure preserving for every $t\in \mathbb{R}$}.$$
We proceed by contradiction and assume that such a $\varphi$ exists.
First, since $\tilde{a}$ is piecewise smooth in $\mathbb{R}^3\setminus ([0,1]^2\times \{0\})$ and thus in particular  $\tilde{a}$ belongs to $BV([0,1]^2\times (0,1])$, we have (cf. \cite{Ambrosio}) that $\varphi$ is uniquely determined (up to a null set) in $[0,1]^2\times (0,1]$. Hence since $a=\tilde{a}$ in $A_1=[0,1]^2\times (0,1]$ we get (cf. Step 2)
that, necessarily, for a.e. $x\in A_1$
\begin{equation}\label{equation.varphi.dans.cube.sup}\varphi(t,x)=(\xi^{(1-x_3)}(t,x_1,x_2),x_3-t)\quad \text{for $t\in [x_3-1,x_3]$}.\end{equation}
Next, since $a=(0,0,-1)$ in $[0,1]^2\times [-1,0)=A_7$, we obviously get that, for every $x\in A_7$
\begin{equation}\label{equation.varphi.dans.cube.inf}\varphi(t,x)=x-t(0,0,1)\quad \text{for $t\in [x_3,1+x_3]$}.\end{equation}
Also, since the third component in identically $-1$ in $A_1\cup A_7$, we trivially obtain that
\begin{equation}\label{equation.3.ieme.comp}\varphi^3(t,x)=x_3-t\quad \text{for every $x\in [0,1]^2\times [-1,1]$ and $t\in [x_3-1,x_3+1]$}.\end{equation}
Now by the group property, we get that for a.e. $x\in A_7$ and $t\in [0,1]$
\begin{equation}\label{equation.group.property}\varphi(t,\varphi(x_3-1,x))=\varphi(t+x_3-1,x).\end{equation}
Combining \eqref{equation.group.property},  \eqref{equation.3.ieme.comp} and \eqref{equation.varphi.dans.cube.sup}, we get that, for a.e. $x\in A_7$ and $t\in [0,1]$
$$\varphi(t+x_3-1,x)=\varphi(t,(y_1,y_2,1))=(\xi^{(0)}(t,y_1,y_2),t)$$
for some $(y_1,y_2)\in [0,1]^2.$
By continuity of $\varphi(\cdot,x)$, combining the previous equation and \eqref{equation.varphi.dans.cube.inf}
we must have
$$\xi^{(0)}(1,y_1,y_2)=(x_1,x_2).$$
Hence, for a.e. $(x_1,x_2)\in ((0,1)\setminus Z)^2$, by  \eqref{gamma.bijection} and \eqref{chi.0.bord.sur.bord}, $y_2$ is the unique number in $(0,1)\setminus Z$ such that $\gamma(y_2)=(x_1,x_2)$ while $y_1\in (0,1)$ can be chosen arbitrarily.

Summarizing, we obtained that, for a.e. $x\in ((0,1)\setminus Z)^2\times [-1,0)$,
$\varphi(t,(x_1,x_2,x_3))$ has necessarily the following form
\begin{equation}\label{formula.for.varphi.non.exist}
\varphi(t,x)=\left\{\begin{array}{cl}
x-t(0,0,1)&\text{for $t\in [x_3,1+x_3]$}\\
(\chi^{(0)}(1+x_3-t,y_1,y_2),x_3-t),&\text{for $t\in [x_3-1,x_3]$}
\end{array}
\right.
\end{equation}
for some $y_1=y_1(x_1,x_2)\in (0,1)$ and where $y_2=y_2(x_1,x_2)\in (0,1)\setminus Z$ is the unique real number such that
$$\gamma(y_2)=(x_1,x_2).$$
We now claim that that
\begin{equation}\label{not.measure.preserving}|\varphi(3/2, ((0,1)\setminus Z)^2\times [-1,-1/2]|=0\end{equation} which implies that $\varphi(3/2,\cdot)$ is not measure preserving whence a contradiction.
From the special structure of the third component of $\varphi$ (cf. \eqref{equation.3.ieme.comp}) \eqref{not.measure.preserving} will be proved once showed that,
for every $x_3\in [-1,-1/2]$ the set
$$M_{x_3}:=\{(\varphi^1(3/2,x_1,x_2,x_3),\varphi^2(3/2,x_1,x_2,x_3)):\text{ $(x_1,x_2)\in ((0,1)\setminus Z)^2$}$$ is a two dimensional null set.
First note that, using \eqref{formula.for.varphi.non.exist},
$$M_{x_3}=\cup_{(x_1,x_2)\in ((0,1)\setminus Z)^2}\{\chi^{(0)}(-x_3-1/2)(y_1(x_1,x_2),y_2(x_1,x_2))\}.$$
Since $\chi^{(0)}(\lambda,\cdot)$ is measure preserving it is enough to show that
$$\cup_{(x_1,x_2)\in ((0,1)\setminus Z)^2}\{(y_1(x_1,x_2),y_2(x_1,x_2))\}$$ is a two dimensional null set.
The latter is obvious since  $(x_1,x_2)\rightarrow y_2(x_1,x_2)$ is one-to-one.
 \end{proof}

In the previous proof we used the following elementary lemma whose proof is omitted.
\begin{lemma} \label{lemma.reflection}
Let $a:\{x_3>0\}\rightarrow \mathbb{R}^3$ be bounded and measurable.
Extend $a$ to $\{x_3<0\}$ by
$$a(x):=-R_3(a(R_3(x))=(-a^1(x_1,x_2,-x_3),-a^2(x_1,x_2,x_3),a^3(x_1,x_2,-x_3)),$$
where $R_3(x_1,x_2,x_3):=(x_1,x_2,-x_3).$
Suppose that for some $x\in \{x_3>0\}$ there exists a map $\varphi(\cdot,x)\in \{x_3\geq 0\}$ defined on $[t_1,t_2]$ with $t_1<t_2$ satisfying
\begin{equation}\label{flow.lemme}\varphi(t,x)=x+\int_0^ta(\varphi(s,x))ds\quad\text{for every $t\in
[t_1,t_2]$}.\end{equation}
Then for $y:=R_3(x)\in \{x_3<0\}$ the map
$$\varphi(t,y):=R_3(\varphi(-t,R_3(y)))\quad \text{$t\in [-t_2,-t_1]$}$$ satisfies
$$\varphi(t,y)=y+\int_0^ta(\varphi(s,y))ds\quad\text{for every $t\in [-t_2,-t_1].$}$$
\end{lemma}

\section{The non autonomous case}

We now establish the two dimensional (non autonomous) version of Theorem \ref{theorem.supp.comp.R.3}.

\begin{theorem}\label{theorem.supp.comp.R.2.t}
 \textbf{Part 1: Non uniqueness.} There exists a 
 compactly supported
 vector field $a \in L^\infty(\mathbb{R}\times \mathbb{R}^2,\mathbb{R}^2)$, 
 such that, $a(t,\cdot)$ is divergence free in $\mathbb{R}^2$ for a.e. $t\in \mathbb{R}$ 
generating two distinct measure preserving flows satisfying the group property.
More precisely, it will be shown the existence of two distinct maps
$\phi,\psi:\mathbb{R}\times \mathbb{R}\times \mathbb{R}^2\rightarrow \mathbb{R}^2$
satisfying, for every $t\in \mathbb{R}$, every $\alpha\neq 1$ and every $x\in \mathbb{R}^2$,
$$
    \phi(t,\alpha,x)= x + \int_0^t a(s+\alpha,\phi(s,\alpha,x)) \ ds,
    \quad \psi(t,\alpha,x)= x + \int_0^t a(s+\alpha,\psi(s,\alpha,x)) \ ds,
$$ 
such that $\phi(t,\alpha,\cdot)$ and $\psi(t,\alpha,\cdot)$ both preserve the Lebesgue
measure for every $\alpha\neq 1$ and $t \in \mathbb{R}$ and such that,
for every $t_1,t_2,\alpha\in \mathbb{R}$ with $\alpha\neq 1$ and $t_2+\alpha\neq 1,$
$$
    \phi(t_1,\alpha+t_2,\phi(t_2,\alpha,x)) 
    = \phi(t_1+t_2,\alpha,x)
    \quad \text{and} \quad \psi(t_1,\alpha+t_2,\psi(t_2,\alpha,x))
    = \psi(t_1+t_2,\alpha,x).
$$

Moreover, there exists a nontrivial
$L^{\infty}([0,\infty)\times \mathbb{R}^2)$ weak solution of
$$
    \partial_t u + \langle a;\nabla_x u \rangle= 0\quad \text{and}\quad u(0,\cdot)=0,
$$
which explicitly means that, for every $h\in C^{\infty}_c([0,\infty)\times \mathbb{R}^2)$
\begin{equation}\label{transport.r2.t}
    \int_{0}^{\infty}\int_{\mathbb{R}^3}  u(t,x) \, \left(\partial_t h(t,x)+\langle a(t,x);\nabla_x
    h(t,x)\rangle\right)dxdt=0.
\end{equation}
\textbf{Part 2: Non existence.}
There exists a divergence free  vector field $\tilde{a} \in L^\infty(\mathbb{R}\times \mathbb{R}^2;\mathbb{R}^2)$
with compact support generating no measure preserving flow satisfying the group property.
\end{theorem}
\begin{remark}
(i) Note that the bounded vector field constructed in \cite{Depauw} (for which the transport equation has two solutions) is periodic in $x$ and hence it does not belong to
$L^p(\mathbb{R}\times \mathbb{R}^2)$ for any $p<\infty$.\smallskip

(ii) The remark \ref{remark.apres.theorme.supp.compact.R3} is also valid for the above theorem.
\end{remark}

\begin{proof} The proof is very similar (and in fact easier) to the one of Theorem \ref{theorem.supp.comp.R.3}. Oversimplifying, the $x_3$ variable in Theorem \ref{theorem.supp.comp.R.3} will play the role of the time in the present proof.\smallskip

\textit{\underline{Step 1.}}
We first define $a(t,x)=b(t,x)$ for $t<1$ and $x\in \mathbb{R}^2$
where $b$ is the vector field constructed in Lemma \ref{lemma.vector.field.b}.
Finally, for $t>1$ and $x\in \mathbb{R}^2$ we let
$$a(t,x):=-a(2-t,x).$$
By a direct application of Lemma \ref{lemma.vector.field.b} we deduce that (in the sense of distributions)
$$\operatorname{div}_x a(t,\cdot)=0\quad \text{for every $t\neq 1$.}
$$
Moreover we observe that
$a\in L^{\infty}(\mathbb{R}\times \mathbb{R}^2;\mathbb{R}^2)$ and that $\operatorname{supp}a\subset [0,2]\times [0,1]^2.$\smallskip

\textit{\underline{Step 2:} A first flow of $a$.}
First for every $\alpha<1$ and $x\in \mathbb{R}^2$ define
$$\phi(t,\alpha,x):=\left\{
\begin{array}{cl}
\chi^{(\alpha)}(t,x)&\text{if $t\leq 1-\alpha$}\\
\chi^{(\alpha)}(2-2\alpha-t,x)&\text{if $t\geq 1-\alpha$}
\end{array}
\right.$$
where $\chi^{(\alpha)}$ is the flow of $(t,x)\rightarrow b(t+\alpha,x)$ exhibited in Lemma \ref{lemma.flow.of.b}.
For $\alpha>1$ define for $x\in \mathbb{R}^2$ and $t\in \mathbb{R}$
$$\phi(t,\alpha,x):=\phi(-t,1-\alpha,x).$$
From Lemma \ref{lemma.flow.of.b} and the fact that $a(t,x)=-a(2-t,x)$ we easily deduce the following properties:
\begin{itemize}
\item Flow of $a$: for every $\alpha\neq 1$, $x\in \mathbb{R}^2$ and $t\in \mathbb{R}$
\begin{equation}\label{flow.a.2}\phi(t,\alpha,x)=x+\int_0^ta(s+\alpha,\phi(t,\alpha,x))ds.
\end{equation}

\item For every $\alpha\neq 1$ and $t+\alpha\neq 1$, $\phi(t,\alpha,\cdot)$ is a bijection from $\mathbb{R}^2$ to $\mathbb{R}^2$ preserving the measure.

\item For every $\alpha\neq 1$ and $t_1,t_2\in \mathbb{R}$ with $\alpha+t_2\neq 1$ we have
\begin{equation}\label{group.prop.phi}\phi(t_2,\alpha+t_2,\phi(t_2,\alpha,x))=\phi(t_1+t_2,\alpha,x).\end{equation}

\item Collapsing of the fibers:
\begin{equation}\label{collopsing.fiber.dim.2}
\phi(1,0,(0,1)\times \{x_2\})\quad \text{is a singleton}
\end{equation}
for every $x_2\in (0,1)\setminus Z$ where
$$Z=\{j2^{-i}|\text{ }0\leq j\leq 2^i,i\geq 1\}.$$
\end{itemize}

\textit{\underline{Step 3:} A different flow for $a$.}
First, for $\alpha<1$ we define, for $x\in\mathbb{R}^2$
$$\psi(t,\alpha,x):=\left\{
\begin{array}{cl}
\phi(t,\alpha,x)&\text{if $t\leq 1-\alpha$}\\
\phi(t,\alpha,x)&\text{if $t\geq 1-\alpha$ and $x\notin A_{\alpha}$}\\
\phi(t,\alpha,\phi(\alpha,0,m(\phi(\alpha,0,\cdot)^{-1}(x))))&\text{if $t\geq 1-\alpha$ and $x\in A_{\alpha}$}\\
\end{array}
\right.$$
where $m(x_1,x_2)=(1-x_1,x_2)$ and
$$A_{\alpha}:=\{\phi(\alpha,0,\cdot)^{-1}\{(0,1)\times ((0,1)\setminus Z)\}.$$
For $\alpha>1$ we define, for every $t\in \mathbb{R}$ and $x\in \mathbb{R}^2,$
$$\psi(t,\alpha,x):=\psi(-t,1-\alpha,x).$$
First from \eqref{group.prop.phi} we deduce that, for every $\alpha\neq 1$ and every $y\in \mathbb{R}^2$
$$\phi(1-\alpha,\alpha,\phi(\alpha,0,y))=\phi(1,0,y).$$
Hence, combining the last equation with \eqref{collopsing.fiber.dim.2}, we get that, for every $\alpha\neq1$ and $x\in A_{\alpha},$
$$\psi(1-\alpha,\alpha,x)=\phi(1-\alpha,\alpha,x).$$
Hence, from \eqref{flow.a.2}, we get that
for every $\alpha\neq 1$, $x\in \mathbb{R}^2$ and $t\in \mathbb{R}$
$$\psi(t,\alpha,x)=x+\int_0^ta(s+\alpha,\psi(t,\alpha,x))ds.$$
Moreover since, $m$ and $\phi(t,\alpha,\cdot)$ are measure preserving and bijections from $\mathbb{R}^2$ onto $\mathbb{R}^2$ for every $\alpha\neq 1$ and $t+\alpha\neq 1$ we get that the same is true for
$\psi(t,\alpha,\cdot)$.
Finally from \eqref{group.prop.phi} we easily that $\psi$ also satisfies the group property: namely $\alpha\neq 1$ and $t_1,t_2\in \mathbb{R}$ with $\alpha+t_2\neq 1$ we have
$$\phi(t_2,\alpha+t_2,\phi(t_2,\alpha,x))=\phi(t_1+t_2,\alpha,x).$$
Also, for every $t\geq 2$, $x_1\in (0,1)\setminus\{1/2\}$ and $x_2\in (0,1)\setminus Z$ we have, since $a(t,\cdot)\equiv 0$ for $t< 0$,
\begin{equation}\label{flow.distinct.2d}
\phi(t,0,x)=(x_1,x_2)\neq (1-x_1,x_2)=\psi(t,0,x).
 \end{equation}
From Steps 2 and 3 we have indeed found two distinct flows of $a$ which are measure preserving and satisfying the group property. \smallskip

\textit{\underline{Step 4:} Non uniqueness for the transport equation.}
For $u_0\in C^{\infty}_c(\mathbb{R}^2)$ define $v,w\in L^{\infty}([0,\infty)\times\mathbb{R}^2)$ by
$$v(t,\cdot):=u_0((\phi(t,0,\cdot))^{-1})\quad \text{and}\quad w(t,\cdot):=u_0((\psi(t,0,\cdot))^{-1}).$$
Proceeding exactly as in Step 7 of the previous proof we have that $v$ and $w$ both solve
$$\frac{\partial u}{\partial t}+
\langle a;\nabla_x u\rangle
=0\quad \text{and}\quad u(0,\cdot)=u_0,$$ in the weak sense.
Choose $u_0\in C^{\infty}_c(\mathbb{R}^2)$ such that $u_0(x)=x_1$ in $(0,1)^2$ and let
$u:=v-w\in L^{\infty}([0,\infty)\times\mathbb{R}^2).$ Then  $u$ is not
identically zero (cf. \eqref{flow.distinct.2d}) and satisfies \eqref{transport.r2.t} which proves the second part of the theorem and concludes the
proof.\smallskip

\textit{\underline{Step 5:} Step 5: Non existence of a flow.}
Define $\tilde{a}:\mathbb{R}\times \mathbb{R}^2\rightarrow \mathbb{R}^2$ by
$$\tilde{a}(t,\cdot):=\left\{
\begin{array}{cl}
a(t,\cdot)&\text{for $t<1$}\\
0&\text{for $t\geq 1$}.
\end{array}
\right.
$$
From the properties of $a$ (cf. Step 1) we directly get that $\tilde{a}$ is bounded, measurable, divergence free and has compact support. Proceeding exactly as in Step 8 of the proof of Theorem \ref{theorem.supp.comp.R.3} we show that there does not exist a measure preserving flow of $\tilde{a}$ satisfying the group property. This proves the last part of the theorem and concludes the proof.
\end{proof}

\section{Appendix}

In the proofs of the previous two theorems we have used the following three lemmas
inspired by \cite{Depauw}.

The first one exhibits two divergence free vector fields in $\mathbb{R}^2$ whose resulting measure preserving flow is a "square" rotation, respectively a "rectangle" rotation, and are the basic bricks to construct the vector field $a$ and $\tilde{a}$ of Theorems \ref{theorem.supp.comp.R.3} and \ref{theorem.supp.comp.R.2.t}.

\begin{lemma}\label{lemma.rotation.carre} 1) Define $c\in L^{\infty}(\mathbb{R}^2;\mathbb{R}^2)$ by
$$c(x):=\left\{
\begin{array}{cl}
(0,8 x_1)& \text{if $|x_2|<|x_1|<1/2$,}\\
(-8 x_2,0)& \text{if $|x_1|<|x_2|<1/2$,}\\
0&\text{elsewhere.}
\end{array}
\right.
$$
Then $\operatorname{div}c=0$ in $\mathbb{R}^2$ in the sense of distributions and the normal component of $c$ is $0$ across $\partial
(-1/2,1/2)^2$.
Additionally  there exists $\xi^c:\mathbb{R}\times\mathbb{R}^2\rightarrow \mathbb{R}^2$ with the following
properties:

(i) for every $x\in \mathbb{R}^2$ and every $t\in \mathbb{R}$
$$\xi^c(t,x)=x+\int_{0}^{t}c(\xi^c(s,x))ds.$$

(ii)  for every $x\in \mathbb{R}^2$ and every $t_1,t_2\in \mathbb{R}$
$$\xi^c(t_1+t_2,x)=\xi^c(t_1,\xi^c(t_2,x)).$$

(iii) for every $t\in \mathbb{R}$, $\xi^c(t,\cdot)$ is a bijection from $\mathbb{R}^2$ onto $\mathbb{R}^2$
 preserving the
Lebesgue measure.

(iv) $\xi^c(t,\cdot)$ is a "square" rotation in $(-1/2,1/2)^2$ of angle $2\pi t$
and the identity outside $(-1/2,1/2)^2$.
In particular
$$\xi^c(1/4,x)=\left\{
\begin{array}{cl}
(-x_2,x_1) &\text{for  $(x_1,x_2)\in (-1/2,1/2)^2$},\\
x &\text{elsewhere}.
\end{array}
\right.
$$

2) Define $d\in L^{\infty}(\mathbb{R}^2;\mathbb{R}^2)$ by
$$d(x):=\left\{
\begin{array}{cl}
(0,4 x_1)& \text{if $|2x_2|<|x_1|<1/2$,}\\
(-8 x_2,0)& \text{if $|x_1|<|2x_2|<1/2$,}\\
0&\text{elsewhere.}
\end{array}
\right.
$$
Then $\operatorname{div} d=0$ in $\mathbb{R}^2$ in the sense of distributions  and the normal component of $d$ is $0$ across $\partial
\left[(-1/2,1/2)\times (-1/4,1/4)\right]$.
Additionally  there exists $\xi^d:\mathbb{R}\times\mathbb{R}^2\rightarrow \mathbb{R}^2$ satisfying the previous points (i)-(iii) with $c$ replaced by $d$. Moreover $\xi^d(t,\cdot)$ is s a "rectangle" rotation in $(-1/2,1/2)\times (-1/4,1/4)$ of angle $2\pi t$
and the identity outside $(-1/2,1/2)\times (-1/4,1/4)$. In particular
$$\xi^d(1/4;x)=\left\{
\begin{array}{cl}
(-2x_2,x_1/2) &\text{for  $(x_1,x_2)\in (-1/2,1/2)\times (-1/4,1/4)$},\\
x &\text{elsewhere}.
\end{array}
\right.
$$
\end{lemma}
\begin{remark} Note that there exist infinitely many flows of $c$ (and $d$); indeed, for example for $c$, one can stay any amount of time once reached the "diagonals" $\{(x_1,x_2)|\text{ }0<|x_1|=|x_2|<1\}$ (where $c$ is identically zero). However, since $c$ and $d$ belong to $BV(\mathbb{R}^2)$ note that $\xi^{c}$, resp $\xi^{d}$, is (up to a null set in $\mathbb{R}\times \mathbb{R}^2$) the unique measure preserving flow of $c$, resp $d$ (cf. \cite{Ambrosio}).
\end{remark}
\begin{proof} \textit{\underline{Step 1:} Proof of 1).}
First we obviously have $\operatorname{div}c=0$ in the four triangles
$$\{-x_2<x_1<x_2,0<x_2<1/2\},\quad  \{-x_2<x_1<x_2,-1/2<x_2<0\}$$
$$\{-x_1<x_2<x_1,0<x_1<1/2\}\quad \text{and} \quad \{-x_1<x_2<x_1,-1/2<x_1<0\}.$$
See Figure 5 for a sketch of $c$.
Moreover since the normal component of $c$ is $0$ across the boundary of each of those four triangles (which contains
$\partial (-1/2,1/2)^2$)
we immediately get that, in the sense of distributions,
$\operatorname{div}c=0$ in $\mathbb{R}^2.$\smallskip

Let $\rho(x):=\max(|x_1|,|x_2|)$. For $x\in \{\rho<1/2\}=(-1/2,1/2)^2$ we write
$x=\rho(x)\theta(x)$ where $\theta$ belongs to the boundary of $(-1/2,1/2)^2$ identified with
$\mathbb{R}/4\mathbb{Z}.$
Then defining $\xi^c:\mathbb{R}\times \mathbb{R}^2\rightarrow \mathbb{R}^2$ by
$$\xi^c(t,x)=x\quad \text{if $x\in \mathbb{R}^2\setminus (-1/2,1/2)^2$}$$ and
$$
\xi^c(t,x)=\xi^c(t,\rho,\theta)=(\rho,\theta+4 t ),\quad \text{for $x\in (-1/2,1/2)^2$,}
$$
\begin{figure}
\begin{center}
\begin{tikzpicture}[scale=0.4]

\draw (-3,3)--(3,3);\draw (-3,-3)--(3,-3);\draw (-3,-3)--(-3,3);\draw (3,-3)--(3,3);
\draw (-3,-3)--(3,3);\draw (-3,3)--(3,-3);
\draw[very thick] (0,2) node[scale=1.5]{$\leftarrow$};
\draw[very thick] (0,-2) node[scale=1.5]{$\rightarrow$};
\draw[very thick] (-2,0) node[scale=1.5]{$\downarrow$};
\draw[very thick] (2,0) node[scale=1.5]{$\uparrow$};
\draw (0,-3.5) node[scale=1]{$c$};

\draw (6,-3)--(6,0);\draw (6,0)--(12,0);\draw (6,-3)--(12,-3);\draw (6,-3)--(12,0);\draw (12,-3)--(12,0);
\draw (6,-3)--(12,0);\draw (6,0)--(12,-3);
\draw[very thick] (9,-3/4) node[scale=1.5]{$\leftarrow$};
\draw[very thick] (9,-9/4) node[scale=1.5]{$\rightarrow$};
\draw[very thick] (7,-1.5) node[scale=0.75]{$\downarrow$};
\draw[very thick] (11,-1.5) node[scale=0.75]{$\uparrow$};
\draw (9,-3.5) node[scale=1]{$d$};
\end{tikzpicture}
\caption{The vector fields $c$ and $d$}
\end{center}
\end{figure}
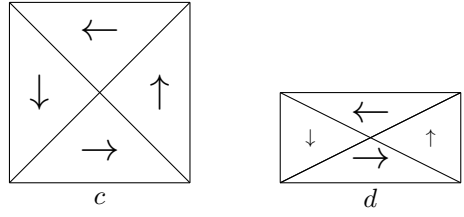
it is easily seen that
$\xi^c$ satisfies all the claimed properties of the lemma. In particular note that $\xi^c(1,x)=x$
hence  $t=1$ corresponds to a rotation of
 $2\pi$ which implies that  $\xi^c(t,\cdot)$ is indeed a square rotation of $2\pi t;$
moreover noting that a "square" rotation of angle $\pi/2$ is the usual rotation of angle $\pi/2$ (observe that this property is only true for integer multiples of $\pi/2$) we get that
$\xi^c(1/4,x)=(x_2,-x_1)$ in $(-1/2,1/2)^2$ (and the identity outside $(-1/2,1/2)^2$).
\smallskip

\textit{\underline{Step 2:} Proof of 2).} The assertions concerning the vector field $d$ are proven exactly as the ones for $c.$
Letting $p:\mathbb{R}^2\rightarrow \mathbb{R}^2$ defined by $p(x_1,x_2):=(x_1,x_2/2)$, note that
$$d(x)=p^{-1}(c(p(x)))\quad \text{for every $x\in \mathbb{R}^2$}.$$ Hence it is elementary to see that
$$\xi^d(t,x):=p^{-1}(\xi^c(t,p(x)))=\left((\xi^c)^1(t,(x_1,2x_2)),\frac{(\xi^c)^2(t,(x_1,2x_2))}{2}\right)$$ satisfies all the wished properties.
\end{proof}

\begin{lemma}\label{lemma.vector.field.b} Let $c$ and $d$ be as in Lemma \ref{lemma.rotation.carre}.
Define $b=b(t,x)\in L^{\infty}((-\infty,1)\times \mathbb{R}^2;\mathbb{R}^2)$ as follows.
First let $b(t,x)\equiv 0$ for $x\notin [0,1]^2$ and $t\in [0,1)$ and for $x\in \mathbb{R}^2$ and $t<0.$ and  Then define it on  $[0,1/2)\times [0,1]^2$ by
$$b(t,x):=\left\{
\begin{array}{cl}
d(x_1-1/2,x_2-1/4) &\text{for $0\leq t<1/4$ and $x\in [0,1]\times [0,1/2], $}\\
d(x_1-1/2,x_2-3/4) &\text{for $0\leq t<1/4$ and $x\in [0,1]\times [1/2,1], $}\\
-c(x_1-1/2,x_2-1/2) &\text{for $1/4\leq t<1/2$ and $x\in [0,1]\times [0,1].$ }
\end{array}
\right.
$$
Define it finally in $[1/2,1)\times [0,1]^2$ by inductively scaling the geometry by a factor $1/2$ (but leaving its range unchanged) in the following way:
For every $i\geq 1$ decompose $[0,1]^2$ into $4^i$ diadic (closed) squares (of size $1/2^{i}$) denoted by $C^i_j,$
$1\leq j\leq 4^i,$
and denote their left lower vertices by $l^{i}_j$. Let also
$$t_i:=\sum_{l=1}^i2^{-l}.$$
Then for every $i\geq 1$ define $b$ in $[t_i,t_{i+1})\times \mathbb{R}^2$ by
$$b(t,x):=\left\{
\begin{array}{cl}
b\big(2^i(t-t_i,x-l^i_j)\big) &\text{for $t\in [t_i,t_{i+1})$ and $x\in C^i_j$, $1\leq j\leq 4^i,$}\\
0 &\text{for $t\in [t_i,t_{i+1})$ and $x\notin [0,1]^2.$}
\end{array}
\right.
$$
Then $b\in L^{\infty}((-\infty,1)\times \mathbb{R}^2;\mathbb{R}^2)$ and for every $t<1$ $\operatorname{div}_x(b(t,\cdot)=0$ in $\mathbb{R}^2$ in the sense of distributions.
Moreover, for every $t<1$ the normal component of $b(t,\cdot)$ is zero across $\partial [0,1]^2$.
\end{lemma}
\begin{proof} First it is clear that $b$ is measurable and bounded in $(-\infty,1)\times \mathbb{R}^2$ once observed that, for every $i\geq 1$ and $1\leq j\leq 4^i,$
$$\|b\|_{L^{\infty}([t_i,t_{i+1})\times C^i_j)}=\|b\|_{L^{\infty}([0,1/2)\times [0,1]^2)}.$$
Since, from Lemma \ref{lemma.rotation.carre}, we know that $\operatorname{div}c=0$ in $(-1/2,1/2)^2$ and that its four normal components are $0$ across $\partial (-1/2,1/2)^2$ and,  similarly for $d$ on the boundary of $(-1/2,1/2)\times (-1/4,1/4)$, we directly deduce from the definition of $b$ that, for every $t<1,$ the normal component of $b(t,\cdot)$ is zero across $\partial [0,1]^2$ and that $\operatorname{div}_xb(t,\cdot)=0$ in $\mathbb{R}^2.$
\end{proof}
\begin{lemma}\label{lemma.flow.of.b} Let $b:(-\infty,1)\times \mathbb{R}^2\rightarrow [0,1]^2$ be the vector field defined in the previous lemma. Then, for every  $z<1$, there exists a measurable map $\chi^{(z)}:(-\infty,1-z]\times \mathbb{R}^2\rightarrow [0,1]^2$ satisfying the following properties:
\begin{itemize}
\item Flow of $b$ shifted by $z:$ for every $x\in \mathbb{R}^2$ then
\begin{equation}\label{flow.of.b}\chi^{(z)}(t,x)=x+\int_0^t b(s+z,\chi^{(z)}(s,x))ds\quad \text{for every $t\in (-\infty,1-z]$}.
\end{equation}
\item For every $t\in (-\infty,1-z),$ $\chi^{(z)}(t,\cdot)$ is a bijection from $\mathbb{R}^2$ onto $\mathbb{R}^2$ preserving the measure.
\item Group property:
for every $x\in \mathbb{R}^2$, $z<1$ and every $t_1,t_2$ with $t_2+z<1$ and $t_1+t_2+z\leq 1$
\begin{equation}\label{group.property.lemma}\chi^{(z+t_2)}(t_1,\chi^{(z)}(t_2,x))=\chi^{(z)}(t_1+t_2,x).\end{equation}
\end{itemize}
Moreover the following properties are fulfilled for $\chi^{(0)}:$
\begin{itemize}
\item Explicit formula for $t=1/2:$ For every $x\in (0,1)^2:$
\begin{equation}\label{action.at.time.t.1.2}\chi^{(0)}(1/2,(x_1,x_2))=\left\{\begin{array}{cl}(x_1/2+\lfloor 2x_2\rfloor/2,2x_2-\lfloor 2x_2\rfloor)&\text{if $x_2\neq 1/2$},\\
(x_2,-x_1+1)&\text{if $x_2=1/2$}\end{array}\right.\end{equation}
where $\lfloor\cdot \rfloor$ stands for the usual integer part.
\item Collapsing property at time $1$:
\begin{equation}\label{collapsing.property}
\chi^{(0)}(1,(0,1)\times \{x_2\})\quad \text{is a singleton}
\end{equation}
for every $x_2\in (0,1)\setminus Z$
where
$$Z:=\left\{\frac{j}{2^{i}}|\text{ }0\leq j\leq 2^i,i\geq 1\right\}.$$

\item Defining $\gamma:(0,1)\setminus Z\rightarrow (0,1)^2$ by
$$\gamma(x_2):=\chi^{(0)}(1,(0,1)\times \{x_2\})$$ we have that
\begin{equation}\label{gamma.bijection}\gamma\quad \text{is a bijection from $(0,1)\setminus Z$ onto $\big((0,1)\setminus Z\big)^2.$}
\end{equation}
Moreover \begin{equation}\label{chi.0.bord.sur.bord}\chi^{(0)}(1,(x_1,x_2))\in (Z\times (0,1))\cup ((0,1)\times Z)\end{equation}
for every $x_1\in (0,1)$ and every $x_2\in Z.$
Furthermore $\gamma$ and $\chi^{(0)}(1,\cdot)$ preserve the measure.
\end{itemize}
\end{lemma}

\begin{proof}
\textit{\underline{Step 1.}}
We first exhibit $\chi^{(0)}$.
First for $x\notin [0,1]^2$ and $t\leq 1$ and for $x\in \mathbb{R}^2$ and $t\leq 0$ we obviously let $\chi^{(0)}(t,x)=x$.
For $x\in [0,1]^2$ and $t\in [0,1]$ we proceed as follows:
We first define $\chi^{(0)}(t,x)$ for $t\in [0,1/4]$ as:
$$\chi^{(0)}(t,x)\\
:=\left\{\begin{array}{cl}
\xi^d(t,x-(1/2,1/4))+(1/2,1/4)&\text{for $x\in [0,1]\times [0,1/2]$}\\
\xi^d(t,x-(1/2,3/4))+(1/2,3/4)&\text{for $x\in [0,1]\times [1/2,1].$}
\end{array}\right.
$$
We then define it for $t\in [1/4,1/2]$ in the following way:
$$\chi^{(0)}(t,x)=\xi^c(-(t-1/4),\chi^{(0)}(1/4,x)).$$
We next define it for $t\in [t_1,t_2]=[1/2,1/2+1/4]$ as follows:
define $y_1:=\chi^{(0)}(t_1,x)$ and let $1\leq j\leq 4$ be such that $y_1\in C^1_j$ and define
\begin{equation}\label{definition.chi.t.1}\chi^{(0)}(t,x):=\frac{1}{2}\chi^{(0)}(2(t-t_1,y_1-l^1_j))+l^1_j.\end{equation}
We then define it by induction for $t\in [t_i,t_{i+1}]$, $i\geq 2$ as follows:
Denote $y_i:=\chi^{(0)}(t_i,x)$ and let $1\leq j\leq 4^i$ be such that $y_i\in C^i_j.$ We then let
\begin{equation}\label{definition.chi.t.i}\chi^{(0)}(t,x):=\frac{1}{2^i}\chi^{(0)}(2^i(t-t_i,y_i-l^i_j))+l^i_j.
\end{equation} Finally we extend $\chi^{(0)}(t,x)$ to $t=1$ by continuity.
We define $\chi^{(z)}$ similarly. It is then a simple exercise to check that the first four properties listed in the statement of the lemma are verified. \smallskip

\textit{\underline{Step 2.}}
We prove \eqref{action.at.time.t.1.2}. First, from Lemma \ref{lemma.rotation.carre}, $\chi^{(0)}(1/2,\cdot)$ consists of a rectangle rotation of angle
$+\pi/2$ in the rectangles
$(0,1)\times (0,1/2)$ and $(0,1)\times (1/2,1)$  followed by a square rotation of angle $-\pi/2$ in
the square $(0,1)^2$ (see  Figure 6).
The rectangle rotation in $(0,1)\times (0,1/2)$, resp. the rectangle rotation in $(0,1)\times (1/2,1)$,  is the
map, using Lemma \ref{lemma.rotation.carre} (ii),
$$v_1(x_1,x_2):=\xi^d(1/2,x_1-1/2,x_2-1/4)+(1/2,1/4)=(-2x_2+1,x_1/2),$$ resp., $$v_2(x_1,x_2):=\xi^d(1/2,x_1-1/2,x_2-3/4)+(1/2,3/4)=(-2x_2+2,x_1/2+1/2).$$
Moreover the square rotation (by the same argument) is easily seen to be the map
$$h(x_1,x_2):=(x_2,-x_1+1).$$
Hence we get
\begin{align*}&\chi^{(0)}(1/2,(x_1,x_2))\\
=&\left\{\begin{array}{cl}
h(v_1(x_1,x_2))=(x_1/2,2x_2)&\text{for $(x_1,x_2)\in (0,1)\times (0,1/2)$}\\
h(v_2(x_1,x_2))=(x_1/2+1/2,2x_2-1)&\text{for $(x_1,x_2)\in (0,1)\times (1/2,1)$,}
\end{array}\right.
\end{align*}
showing the first equation in \eqref{action.at.time.t.1.2}.
When $x_2=1/2$ both rectangle rotations act trivially ($(x_1,1/2)$ is sent to $(x_1,1/2)$)
while the square rotation sends $(x_1,1/2)$ to $(1/2, -x_1-1)$ which shows the second equation in \eqref{action.at.time.t.1.2}.
\smallskip

\textit{\underline{Step 3.}} We now prove \eqref{collapsing.property}.
From \eqref{action.at.time.t.1.2} we have in particular that
 for every $x_2\in (0,1)\setminus \{1/2\}$,  the fiber $(0,1)\times \{x_2\}$ is send
by $\chi^{(0)}(1/2,\cdot)$ to the
fiber of length $1/2$
$$(m_1(x_2),1/2+m_1(x_2))\times\{n_1(x_2)\}$$
where $m_1(x_2):=1/2\lfloor 2x_2\rfloor\in \{0,1/2\}$ and $n_1(x_2):=2x_2-\lfloor 2x_2\rfloor$.
Trivially $n_1(x_2)$ does not belong $Z$ whenever $x_2$ does not belong to $Z$ where we recall that $$Z=\left\{\frac{j}{2^{i}}|\text{ }0\leq j\leq 2^i,i\geq 1\right\}.$$
Next, using \eqref{definition.chi.t.1}, a direct calculation gives
that, for every $x_2\in (0,1)\setminus \{1/4,1/2,3/4\}$, $\chi^{(0)}(t_2,\cdot)$ sends $(0,1)\times \{x_2\}$ to the fiber of length $1/4$
$$(m_2(x_2),1/4+m_2(x_2))\times \{n_2(x_2)\}$$
where $m_2(x_2)\in \{0,1/4,1/2,3/4\}$ and where
$$n_2(x_2)\in (0,1)\setminus Z\quad \text{whenever $x_2\in (0,1)\setminus Z$}.$$
Proceeding by induction, we obtain that, for every $i\geq 2$ and for every $x_2\in (0,1) \setminus Z,$
$$\chi^{(0)}\left(t_i;(0,1)\times\{x_2\}\right)=(m_i(x_2),2^{-i}+m_i(x_2))\times
\{n_i(x_2)\}$$
for some  $m_i(x_2)\in \{j2^{-i}|\text{ }0\leq j<2^i\}$ and $n_i(x_2)\in (0,1)\setminus Z.$
Letting $i$ going to $\infty$ we eventually obtain \eqref{coll.2d}.
\smallskip

\textit{\underline{Step 4.}}
First thanks to \eqref{collapsing.property} $\gamma$ is well defined.
Writing every $x_2\in (0,1)$ in base four, i.e
$$x_2=0,\alpha_1\alpha_2\cdots$$ with $\alpha_i\in \{0,1,2,3\}$ and
$$x_2=\sum_{i=1}^{\infty}\alpha_i4^{-i}$$ we get that
$$Z=\{x_2\in (0,1):\text{ $\exists I$ such that $\alpha_i=0$ for every $i\geq I$ or $\alpha_i=3$ for every $i\geq I$.}\}$$
Writing $\gamma(x_2)=(\gamma^1(x_2),\gamma^2(x_2))$ is base 2 i.e
$$\gamma^j(x_2)=0,\beta^j_1\beta^j_2\cdots$$ with $\beta^j_i\in \{0,1\}$ and
$$\gamma^j(x_2)=\sum_{i=1}^{\infty}\beta^j_i2^{-i}$$
we easily get by induction (see Figure 6 for $i=1$) that the $\beta^j_i$ obey the following rule
$$\beta^1_i=\left\{\begin{array}{cl}
0&\text{if $\alpha_i\in \{0,2\}$}\\
1&\text{if $\alpha_i\in \{1,3\}$}\\
\end{array}
\right.
\quad \text{and}\quad\beta^2_i=\left\{\begin{array}{cl}
0&\text{if $\alpha_i\in \{0,1\}$}\\
1&\text{if $\alpha_i\in \{2,3\}.$}\\
\end{array}
\right.
$$
From these two formulas we get at once that $\gamma$ is one-to-one on $(0,1)\setminus Z$.
Moreover, noting that
$$Z=\{y\in (0,1):\text{ $\exists I$ such that $\beta_i=0$ for every $i\geq I$ or $\beta_i=1$ for every $i\geq I$}\}$$
we get that, by the characterization of $Z$ (in base 2 and 4) and by the formula for $\gamma$,
$$\gamma((0,1)\setminus Z)=\big((0,1)\setminus Z\big)^2$$ proving \eqref{gamma.bijection}.

Next noting that $\chi^{(0)}(1/2,\cdot)$ is the identity on $\partial [0,1]^2$ and sends (using \eqref{action.at.time.t.1.2})
$$(0,1)\times \{1/4,1/2,34\}\quad \text{to}\quad ((0,1)\times \{1/2\})\cup (\{1/2\}\times (0,1))$$ we easily get \eqref{chi.0.bord.sur.bord} proceeding by induction.

We finally establish the claim concerning the preservation of the measure.
First, by definition of $\gamma$, $\chi^{(0)}(1,\cdot)$ preserves the measure (from $(0,1)^2$ to $(0,1)^2$) if and only if $\gamma$ preserves the measure (from $(0,1)$ to $(0,1)^2$).
Then we get that $\chi^{(0)}(1,\cdot)$ is measure preserving as the pointwise limit of the measure preserving maps $\chi^{(0)}(1-1/n,\cdot)$.
One other direct way to prove the claim is to notice that (using the formula for $\gamma$)
for every $i\geq 1$ and every $0\leq k<4^i$
the "interval" $\{x_2\in (0,1)\setminus Z: k/4^i<x_2<(k+1)/4^i$ of length $4^{-i}$ is sent by $\gamma$ to the "square"
$$\big((l_{x_2},l_{x_2}+2^{-i})\setminus Z\big)\times \big((m_{x_2},m_{x_2}+2^{-i})\setminus Z\big)$$ of area $4^{-i}$ for some $l_{x_2},m_{x_2}\in Z$; hence by bijectivity  of $\gamma$ we get that $\gamma$ is measure preserving.

\end{proof}

\begin{figure}
\begin{center}
\begin{tikzpicture}[scale=0.6]

\draw[very thick] (-8,8)--(-4,8);\draw[very thick] (-8,6)--(-4,6);\draw[very thick] (-8,4)--(-4,4);
\draw (-8,7)--(-4,7);\draw (-8,5)--(-4,5);
\draw[very thick](-8,8)--(-8,4);\draw[very thick] (-4,8)--(-4,4);
\draw[very thick] (-6,5) node[scale=1]{$\curvearrowleft$};
\draw (-6,5) node[scale=0.5]{$\bullet$};\draw (-5.55,5.2) node[scale=0.5]{$+\pi/2$};
\draw (-6,4.5) node[scale=1]{$1$};\draw (-6,5.5) node[scale=1]{$2$};
\draw (-6,6.5) node[scale=1]{$3$};\draw (-6,7.5) node[scale=1]{$4$};
\draw[very thick] (-6,7) node[scale=1]{$\curvearrowleft$};
\draw (-6,7) node[scale=0.5]{$\bullet$};\draw (-5.55,7.2) node[scale=0.5]{$+\pi/2$};

\draw (-3,6) node[scale=1]{$\rightarrow$};

\draw[very thick] (-2,8)--(2,8);\draw (-2,6)--(2,6);\draw[very thick] (-2,4)--(2,4);
\draw[very thick] (-2,4)--(-2,8);\draw (0,4)--(0,8);\draw[very thick] (2,4)--(2,8);
\draw[very thick] (0,6) node[scale=1.5]{$\curvearrowright$};
\draw (0,6) node[scale=0.5]{$\bullet$};\draw (0.4,6.3) node[scale=0.5]{$-\pi/2$};
\draw (-1,7) node[scale=1]{$4$};\draw (1,7) node[scale=1]{$3$};
\draw (-1,5) node[scale=1]{$2$};\draw (1,5) node[scale=1]{$1$};

\draw (3,6) node[scale=1]{$\rightarrow$};

\draw (4,8)--(8,8);\draw (4,6)--(8,6);\draw (4,4)--(8,4);
\draw (4,4)--(4,8);\draw (6,4)--(6,8);\draw (8,4)--(8,8);
\draw (5,7) node[scale=1]{$2$};\draw (7,7) node[scale=1]{$4$};
\draw (5,5) node[scale=1]{$1$};\draw (7,5) node[scale=1]{$3$};

\end{tikzpicture}
\caption{The action of $\xi^{(0)}(1/2,\cdot)$}
\end{center}
\end{figure}
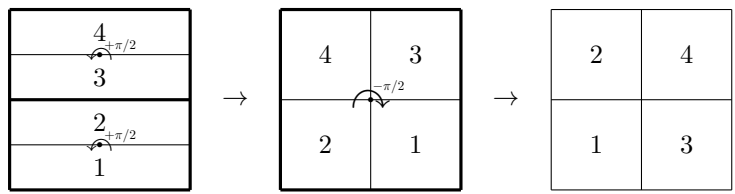

\section*{Acknowledgements}

The author Olivier Kneuss is supported by the 
CNPq-Science without Borders, BJT 2014, through the grant
400378/2014-0. 
The author Wladimir Neves is partially supported by
CNPq through the grant 308652/2013-4, and by FAPERJ 
(Cientista do Nosso Estado) 
through the grant E-26/203.043/2015.


\end{document}